\documentclass{amsart}
\usepackage{amssymb,mathrsfs,amscd,graphicx, array, mathtools, color}
\allowdisplaybreaks

\usepackage[nocompress, noadjust]{cite}  %Sorting citations. See
\usepackage{enumitem}
\usepackage{tikz}
\usetikzlibrary{matrix,arrows,cd}
\usepackage{bbm}

\usepackage[
  pdfencoding=unicode, % or unicode
  psdextra,
]{hyperref}
\usepackage{letltxmacro}
\LetLtxMacro{\oldsqrt}{\sqrt}
\renewcommand{\sqrt}[2][]{\,\oldsqrt[#1]{#2}\,}
\makeatletter
\def\@tocline#1#2#3#4#5#6#7{\relax
  \ifnum #1>\c@tocdepth % then omit
  \else
    \par \addpenalty\@secpenalty\addvspace{#2}%
    \begingroup \hyphenpenalty\@M
    \@ifempty{#4}{%
      \@tempdima\csname r@tocindent\number#1\endcsname\relax
    }{%
      \@tempdima#4\relax
    }%
    \parindent\z@ \leftskip#3\relax \advance\leftskip\@tempdima\relax
    \rightskip\@pnumwidth plus4em \parfillskip-\@pnumwidth
    #5\leavevmode\hskip-\@tempdima
      \ifcase #1
       \or\or \hskip 1em \or \hskip 2em \else \hskip 3em \fi%
      #6\nobreak\relax
    \dotfill\hbox to\@pnumwidth{\@tocpagenum{#7}}\par
    \nobreak
    \endgroup
  \fi}
\makeatother
\usepackage{stmaryrd}

\makeatletter

\def\greekbolds#1{%
 \@for\next:=#1\do{%
    \def\X##1;{%
     \expandafter\def\csname V##1\endcsname{\boldsymbol{\csname##1\endcsname}}
     }
   \expandafter\X\next;
  }
}

\greekbolds{alpha,beta,iota,gamma,lambda,nu,eta,Gamma,varsigma}

\def\make@bb#1{\expandafter\def
  \csname bb#1\endcsname{{\mathbb{#1}}}\ignorespaces}

\def\make@bbm#1{\expandafter\def
  \csname bb#1\endcsname{{\mathbbm{#1}}}\ignorespaces}

\def\make@bf#1{\expandafter\def\csname bf#1\endcsname{{\bf
      #1}}\ignorespaces} 

\def\make@gr#1{\expandafter\def
  \csname gr#1\endcsname{{\mathfrak{#1}}}\ignorespaces}

\def\make@scr#1{\expandafter\def
  \csname scr#1\endcsname{{\mathscr{#1}}}\ignorespaces}

\def\make@cal#1{\expandafter\def\csname cal#1\endcsname{{\mathcal
      #1}}\ignorespaces} 
\def\do@Letters#1{#1A #1B #1C #1D #1E #1F #1G #1H #1I #1J #1K #1L #1M
                 #1N #1O #1P #1Q #1R #1S #1T #1U #1V #1W #1X #1Y #1Z}
\def\do@letters#1{#1a #1b #1c #1d #1e #1f #1g #1h #1i #1j #1k #1l #1m
                 #1n #1o #1p #1q #1r #1s #1t #1u #1v #1w #1x #1y #1z}
\do@Letters\make@bb   \do@letters\make@bbm
\do@Letters\make@cal  
\do@Letters\make@scr 
\do@Letters\make@bf \do@letters\make@bf   
\do@Letters\make@gr   \do@letters\make@gr
\makeatother

\newcommand{\abs}[1]{\lvert #1 \rvert}
\newcommand{\zmod}[1]{\mathbb{Z}/ #1 \mathbb{Z}}

\newcommand{\wh}{\widehat}

\newcommand{\sg}{\mathrm{sg}}
\newcommand{\scc}{\mathrm{sc}} %\sc is already defined for small
\newcommand{\op}{\mathrm{op}}

\newcommand{\bsh}{\backslash}

\newcommand{\whI}{\widehat{I}}
\newcommand{\whJ}{\widehat{J}}
\newcommand{\whD}{\widehat{\mathfrak{D}}}
\newcommand{\whE}{\widehat{\mathcal{E}}}

\newcommand{\whF}{\widehat{F}}

\newcommand{\whK}{\widehat{K}}
\newcommand{\whO}{\widehat{O}}

\newcommand{\wcO}{\widehat{\mathcal{O}}}
\newcommand{\wbZ}{\widehat{\mathbb{Z}}}

\DeclareMathOperator{\Emb}{Emb}
\DeclareMathOperator{\SG}{SG}
\DeclareMathOperator{\SCl}{SCl}

\def\op{\mathrm{op}}

\DeclareMathSymbol{\twoheadrightarrow} {\mathrel}{AMSa}{"10}

\DeclareMathOperator{\Tp}{Tp}
\DeclareMathOperator{\Cl}{Cl}

\DeclareMathOperator{\Hom}{Hom}

\DeclareMathOperator{\Gal}{Gal}

\DeclareMathOperator{\Nm}{N}  %%% norm
\DeclareMathOperator{\Nr}{Nr}

\newcommand{\Z}{\mathbb Z}
\newcommand{\Q}{\mathbb Q}

\renewcommand{\grD}{D}
\renewcommand{\whD}{\widehat{D}}

\newcounter{thmcounter} 
\numberwithin{thmcounter}{section}  % place this command in different
\newtheorem{thm}[thmcounter]{Theorem}
\newtheorem{lem}[thmcounter]{Lemma}
\newtheorem{cor}[thmcounter]{Corollary}
\newtheorem{prop}[thmcounter]{Proposition}
\theoremstyle{definition}
\newtheorem{defn}[thmcounter]{Definition}

\newtheorem{ex}[thmcounter]{Example}

\newtheorem{rem}[thmcounter]{Remark}

\numberwithin{equation}{section}
\numberwithin{figure}{section}
\numberwithin{table}{section}

\newtheoremstyle{notitle}  % this product a paragraph that's
  {}% measure of space to leave above the theorem. E.g.: 3pt
  {}% measure of space to leave below the theorem. E.g.: 3pt
  {\itshape}% name of font to use in the body of the theorem
  {}% measure of space to indent
  {}% name of head font
  {\ }% punctuation between head and body
  {.5em}% space after theorem head; " " = normal interword space
  {}% Man
\theoremstyle{notitle}

 \title[Optimal spinor selectivity]{Optimal spinor  selectivity for 
 quaternion orders}
 \author{Jiangwei Xue}

\address{(Xue) Collaborative Innovation Center of Mathematics, School of
  Mathematics and Statistics, Wuhan University, Luojiashan, 430072,
  Wuhan, Hubei, P.R. China}   
\address{(Xue) Hubei Key Laboratory of Computational Science (Wuhan
  University), Wuhan, Hubei,  430072, P.R. China.}

\email{xue\_j@whu.edu.cn}

\author{Chia-Fu Yu}

\address{(Yu) Institute of Mathematics,
  Academia Sinica and NCTS, Astronomy-Mathematics
  Building, No. 1, Sec. 4, Roosevelt Road, Taipei 10617, TAIWAN.}

\email{chiafu@math.sinica.edu.tw}

\begin{document}
\date{\today} 
\begin{abstract}
 Let $D$ be a quaternion algebra over a number
  field $F$, and $\scrG$ be an arbitrary genus of $O_F$-orders of full
  rank in $D$. Let $K$ be a quadratic field extension of $F$ that
  embeds into $D$, and $B$ be an $O_F$-order in $K$ that can be
  optimally embedded into some member of $\scrG$.   We provide a necessary and sufficient condition for $B$ to be
  optimally spinor selective for the genus $\scrG$, which generalizes
  previous existing optimal selectivity criterions for Eichler orders as given by 
  Arenas, Arenas-Carmona and Contreras, and by Voight independently. 
  %M.~Arenas et
  %al. \cite{M.Arenas-et.al-opt-embed-trees-JNT2018} and by Voight
  %\cite[\S31]{voight-quat-book}.  
  This allows us to obtain a refinement of the classical trace formula for optimal
  embeddings, which  will be called the spinor trace formula.
  When $\scrG$ is a genus of Eichler orders, we  extend Maclachlan's relative
  conductor formula %\cite{Maclachlan-selectivity-JNT2008} 
  for optimal selectivity from
  Eichler orders of square-free levels to all Eichler orders. 
 \end{abstract}
 
\maketitle
%\tableofcontents   %% this add the table of contents after the abstract.
%%%%%%%%%%            How to add to content line              %%%%%%%%%%%
%\addcontentsline{toc}{chapter}{\protect\numberline{}Appendix}

%%%%%%%%%%%%%%%%%              START HERE        %%%%%%%%%%%%%%%%%%%%%%%%
%%%%%%%%%%%%%%%%%%%%%%%%%%%%%%%%%%%%%%%%%%%%%%%%%%%%%%%%%%%%%%%%%%%%%%%%%

%\linenumbers

% Hence $O_{F,+}^\times=\dangle{\varepsilon}$ if $p\equiv
% 3\pmod{4}$, and $O_{F,+}^\times=\dangle{\varepsilon^2}$ otherwise.
% On
% the other hand, $O_F^{\times 2}=\dangle{\varepsilon^2}$ for all $p$,
% so we have 
% \begin{equation}
%   \label{eq:29}
% \abs{\gru}=
% \begin{cases}
%   1 &\quad \text{if } p\not\equiv 3\pmod{4};\\
%   2 &\quad \text{if } p\equiv 3\pmod{4}.
% \end{cases}
% \end{equation}

\section{Introduction}

Let $F$ be a number field, $K$ be a quadratic field extension of $F$,
and $D$ be a quaternion $F$-algebra.  According to the
Hasse-Brauer-Noether-Albert Theorem \cite[Theorem~III.3.8]{vigneras},
$K$ embeds into $D$ over $F$ if and only if there is no place of $F$ that is
simultaneously ramified in $D$ and split in $K$.  The optimal (spinor)
selectivity question studies an integral refinement of this 
theorem.  To state the question and the answers  precisely, we set up some
notations and definitions.

Throughout this paper, we assume that $K$ embeds into $D$ over $F$. Orders in $D$ (resp. $K$) refer
exclusively to
 $O_F$-orders of
\emph{full rank} in $D$ (resp.~$K$), where $O_F$ is the
ring of integers of $F$ as usual. 
Two orders $\calO$ and $\calO'$ in $D$ are said to be \emph{locally isomorphic}
if their
$\grp$-adic completions $\calO_\grp$ and $\calO_\grp'$ are isomorphic
at every finite prime $\grp$ of $F$.  This defines an equivalence
relation on the set of orders in $D$,  and an equivalence class is
called a \emph{genus of orders in $D$}. For example, all maximal orders
of $D$ form a single genus. The genus represented by an (arbitrary) order
$\calO$ will be denoted by $\scrG(\calO)$. 
Similarly, two orders $\calO$ and $\calO'$ in $D$ are said to be
\emph{of the same type} if they are $O_F$-isomorphic, or equivalently,
if there exists $x\in D^\times$ such that $\calO'=x\calO x^{-1}$. The
type of $\calO$ will be denoted by $[\calO]$,
i.e.~$[\calO]:=\{x\calO x^{-1}\mid x\in D^\times\}$. 
 Each genus $\scrG$ of orders in
$D$ is subdivided into \emph{finitely many} types \cite[\S17.1]{voight-quat-book}, and we write
$\Tp(\scrG)$ for the finite set of types in $\scrG$.  If
$\scrG=\scrG(\calO)$, then we  put $\Tp(\calO):=\Tp(\scrG(\calO))$ and regard it as a pointed set with $[\calO]$ as
the base point.

 %Given $\calO\in \scrG$, we  put $\Tp(\calO):=\Tp(\scrG)$ and regard it as  a pointed set with $[\calO]$
 % as its base point.

% state the main result, we need to set up quite some notation.

% A set $\scrG$ consisting of $O_F$-orders (of
% full rank) in $D$ is called a \emph{genus} if
% \begin{enumerate}[label=(\roman*)]
% \item for any two orders $\calO, \calO'\in \scrG$, their $\grp$-adic completions $\calO_\grp$ and $\calO_\grp'$ are isomorphic
%   at every finite prime $\grp$ of $F$,
% \item $\scrG$ is maximal with respect to inclusion among the sets of
%   $O_F$-orders satisfying (i).   
% \end{enumerate}

Given an order $B$ in $K$ and an order $\calO$ in $D$, we write  $\Emb(B, \calO)$ for the set of 
\emph{optimal embeddings} of $B$ into $\calO$, that is
\begin{equation}
  \label{eq:22}
  \Emb(B, \calO):=\{\varphi\in \Hom_F(K, \grD)\mid \varphi(K)\cap \calO=\varphi(B)\}.
\end{equation}
The unit group $\calO^\times$ acts on $\Emb(B, \calO)$ from the right by conjugation:
$\varphi\mapsto u^{-1}\varphi u$ for every $u\in \calO^\times$. It is
well known \cite[\S30.3--30.5]{voight-quat-book} that the number of of
orbits is finite both in the global and local cases, so we put
\begin{equation}
m(B, \calO, \calO^\times):=\abs{\Emb(B, \calO)/\calO^\times}, \quad
m(B_\grp, \calO_\grp, \calO_\grp^\times):=\abs{\Emb(B_\grp,
  \calO_\grp)/\calO_\grp^\times}. 
\end{equation}
% Clearly, $  m(B_\grp, \calO_\grp, \calO_\grp^\times)$ depends only on
% the genus $\scrG$ and not on the choice of $\calO\in \scrG$.  Since
% $\scrG$ is often clear from the context, so we simply put 
% \begin{equation}
%   m_\grp(B):=m(B_\grp, \calO_\grp, \calO_\grp^\times)
% \end{equation}

% Here $B_\grp$ (resp.~$\calO_\grp$) denotes the $\grp$-adic
% completion of $B$ (resp.~$\calO$) at a finite prime $\grp$ of $F$. 

Clearly, whether $\Emb(B, \calO)=\emptyset$ or not depends only on the
type $[\calO]$. Moreover, if $\Emb(B, \calO)\neq \emptyset$, then 
\begin{equation}
  \label{eq:3}
  \Emb(B_\grp,
\calO_\grp)\neq \emptyset \quad\text{for every finite prime $\grp$ of $F$}.\tag{$*$}
\end{equation}
We aim to study the converse of this implication. For the rest of
this paper, we assume that $\scrG$ is a genus of orders in $D$ such
that condition (\ref{eq:3}) holds for $B$ and  for
every $\calO\in \scrG$. This guarantees the existence of
$[\calO_0]\in \Tp(\scrG)$ such that $ \Emb(B, \calO_0)\neq \emptyset$
by \cite[Corollary~30.4.18]{voight-quat-book}.  The optimal
selectivity question further asks whether assumption \eqref{eq:3} implies that
$\Emb(B, \calO)\neq \emptyset$ for every type $[\calO]\in \Tp(\scrG)$.

% More precisely, suppose that $\scrG$ is a genus of
% orders in $D$ such that condition (\ref{eq:3}) holds for $B$ and
% $\scrG$ (that is, for every $\calO\in \scrG$).

The optimal selectivity question is best studied when  $D$ satisfies the Eichler condition (i.e.~there
exists an archimedean place of $F$ that does not ramify in
$D$), for otherwise one can easily construct examples that answer
negatively to this question (See Example~\ref{ex:Dpinf}).    However, the answer  can still be negative even if one assumes the 
Eichler condition,  as discovered by Chinburg and Friedman
\cite{Chinburg-Friedman-1999}.  Indeed, it was them who first introduced the notion of ``selectivity''.

% However, in the most general setting, one cannot expect more than
% this.  For example, let $D_{p, \infty}$ be the unique quaternion
% algebra over $\bbQ$ ramified precisely at a prime $p\in \bbN$ and
% $\infty$, and let $\scrG_{p, \infty}$ be the genus of maximal orders
% in $D_{p, \infty}$.  From \cite[Proposition~V.3.2]{vigneras}, if
% $p\equiv 3\pmod{4}$, then there is a unique member
% $[\calO_0]\in \Tp(\scrG_{p, \infty})$ such that
% $\Emb(\Z[\sqrt{-1}], \calO_0)\neq \emptyset$, while
% $\abs{\Tp(\scrG_{p, \infty})}$ tends to infinity as $p$ goes to
% infinity. 

\begin{defn}[{\cite[Definition~31.1.5]{voight-quat-book}}]
  Suppose that $D$ satisfies the Eichler condition. If $\Emb(B,
  \calO')=\emptyset$ for some  but
  not all\footnote{The ``not all'' part is automatic since we assume condition
    (\ref{eq:3}) throughout.}  orders 
  $\calO'\in \scrG$, then we say
  that $B$ is \emph{optimally selective} for $\scrG$. 
\end{defn}

Naturally, one further asks exactly when optimally selectivity
happens, and how to  characterize  those types of orders
$[\calO]\in \Tp(\scrG)$  that do admit  optimal embeddings from
$B$.  A lot research has been carried out on these topics following Chinburg and
Friedman's pioneering work.  Maclachlan 
\cite{Maclachlan-selectivity-JNT2008} first obtained an optimal
selectivity theorem for Eichler orders of
square-free levels. Independently, Arenas et~al.\
\cite{M.Arenas-et.al-opt-embed-trees-JNT2018} and Voight
\cite[Chapter~31]{voight-quat-book} removed the square-free condition
and obtained theorems for Eichler orders of
arbitrary levels. See
\cite[\S31.7.7]{voight-quat-book} for a brief historical note on the
contributions of Chan-Xu \cite{Chan-Xu-Rep-Spinor-genera-2004},
Guo-Qin \cite{Guo-Qin-embedding-Eichler-JNT2004}, Linowitz
\cite{Linowitz-Selectivity-JNT2012}, Arenas-Carmona
\cite{Arenas-Carmona-Spinor-CField-2003, Arenas-Carmona-2013,
  Arenas-Carmona-cyclic-orders-2012, Arenas-Carmona-Max-Sel-JNT-2012},
and many others.

Central to the theory of selectivity is a class field $\Sigma_\scrG$
attached to the genus $\scrG$ (See Definition~\ref{defn:spinor-field}). It is an abelian
extension of $F$ of exponent $2$ satisfying $[\Sigma_\scrG:
F]=\abs{\Tp(\scrG)}$ (still assuming that $D$ satisfies the Eichler
condition, cf.~\eqref{eq:8} and Remark~\ref{rem:Eichler-con}). 
For the reader's convenience, we reproduce Voight's
formulation of optimal selectivity theorem for Eichler orders.  
\begin{thm}[{\cite[Theorem~31.1.7]{voight-quat-book}}]\label{thm:Voight31}
  Keep the assumption that $D$ satisfies the Eichler condition.  Let
  $\scrG$ be a genus of Eichler orders. Then the following statements
  hold.
  \begin{enumerate}[label=(\roman*)]
\item $B$  is optimally selective for the genus $\scrG$ if and only if
  $K\subseteq \Sigma_\scrG$.
  \item If $B$ is optimally selective for the genus $\scrG$, then
    $\Emb(B, \calO)\neq \emptyset$ for exactly half of the types
    $[\calO]\in \Tp(\scrG)$. 
  \item In all cases, $m(B, \calO, \calO^\times)=m(B, \calO',
    \calO'^\times)$ for $\calO, \calO'\in \scrG$ whenever both sides
    are nonzero. 
      \end{enumerate}
    \end{thm}

% If $B$ is optimally selective for $\scrG$,  then the types
% $[\calO]\in \Tp(\scrG)$ with $\Emb(B, \calO)\neq \emptyset$ 

In this paper, we obtain a number of generalizations of the above
theorem. Our main theorem (Theorem~\ref{thm:selectivity}) gives  the
necessary and sufficient conditions for   optimal  selectivity that applies to any arbitrary genus of
  quaternion orders. The global condition $K\subseteq \Sigma_\scrG$ is
  still necessary, but the sufficiency condition requires additional
  local considerations at  finitely many places where $\calO$ has
  Eichler invariant zero (See Definition~\ref{defn:eichler-invariant}).
  On one hand, this criterion of optimal selectivity shows that Theorem~\ref{thm:Voight31}
  extends verbatim to orders in $D$ that have nonzero
  Eichler invariants at all finite places of $F$ (which include the
  Eichler orders as a proper subclass). On the other hand,
  it paves the way for the discovery in \cite{peng-xue:select} of 
 pairs $(B, \scrG)$ for which the condition
  $K\subseteq \Sigma_\scrG$ is insufficient for optimal selectivity.

When $B$ is optimally selective for the
  genus $\scrG$, there is a method determining all the types
  $[\calO]\in \Tp(\scrG)$ with $\Emb(B, \calO)\neq \emptyset$ using the
  Artin symbol, provided one of such type is known prior.  See
  \cite[Lemma~2.2]{M.Arenas-et.al-opt-embed-trees-JNT2018},
  \cite[\S31.1.9]{voight-quat-book}, or Theorem~\ref{thm:selectivity} of the present
  paper.  However, in practice it might not be easy to construct
  explicitly an order 
 $\calO_0\in \scrG$ with
  $\Emb(B, \calO_0)\neq \emptyset$.  In the study of optimal
  selectivity for Eichler orders of square-free levels, Maclachlan
  \cite{Maclachlan-selectivity-JNT2008} had the novel idea of relating the
  optimal selectivity of two distinct orders $B, B'\subset K$ via
  their relative conductor (See (\ref{eq:120})). For
  example, to know whether $\Emb(B,\calO)=\emptyset $ or not for an
  Eichler order $\calO$ of square-free level, we can
  choose a suitable $F$-embedding $\varphi: K\to D$ and put
  $B':=\varphi^{-1}(\calO)$. Then  the  relative conductor $\grf(B'/B)$ encodes the desired
  information for $\Emb(B,\calO)$. In Theorem~\ref{thm:maclachlan}, we shall
  remove the square-free condition and extend Maclachlan's result to
  all Eichler orders. % We also provide an example showing that this
  % approach is unlikely to be extended further  to more general orders
  % beyond the Eichler ones.  

%   As an application, if 
%   is an  and $\varphi(B')=\varphi(K)\cap \calO$, then 
%  can be determined using the
% Artin symbol of the relative conductor $\grf(B'/B)$. 

  % Once we know one type $[\calO_0]\in \Tp(\scrG)$ with
  % $\Emb(B, \calO_0)\neq \emptyset$ is known, then we can determine the
  % rest of the 

  % There is a method to determine  all the types
  % This method requires prior
  % knowledge of at least one 
  Built on the previous work of Vign\'eras \cite[\S III.5]{vigneras},
  Voight gives an explicit formula for $m(B, \calO, \calO^\times)$ in
  \cite[Corollary~31.1.10]{voight-quat-book} (cf.~\eqref{eq:11}) by combing part (iii) of
  Theorem~\ref{thm:Voight31} together with \emph{the trace formula}
  \cite[Theorem~III.5.11]{vigneras}
  \cite[Theorem~30.4.7]{voight-quat-book}. Here the Eichler condition
  is essential as remarked by him in
  \cite[Remark~31.6.2]{voight-quat-book}. On the other hand, when
  studying certain class number formulas attached to orders in totally
  definite quaternion algebras (i.e.~quaternion algebras that do not
  satisfy the Eichler condition), the present authors are confronted
  with the task of generalizing the Vign\'eras-Voight formula to the
  totally definite case. As a prerequisite, we need a version of the
  optimal selectivity theorem that applies to the totally definite
  case. This in itself is not new. In \cite[Remark,
  p.~99]{M.Arenas-et.al-opt-embed-trees-JNT2018}, M.~Arenas et
  al.~remark that the first two parts of Theorem~\ref{thm:Voight31}
  still hold as soon as the words ``conjugate class'' and
  ``conjugate'' are replaced by ``spinor genera'' and by ``spinor
  equivalence'' respectively.  We shall take this approach and
  formulate all our results in terms of \emph{optimal spinor
    selectivity} (see Definition~\ref{defn:oss}). Such a formulation 
  allows us to get a similar result as part (iii) of
  Theorem~\ref{thm:Voight31}, except that both sides of the equality
  are of the form $\sum_i m(B, \calO_i, \calO_i^\times)$, where the
  $\calO_i$'s range over the left orders of right $\calO$-ideal
  classes within one spinor class (See (\ref{eq:132})).  This leads to a refinement of the
  classical trace formula, which will be called the \emph{spinor trace
    formula} in Proposition~\ref{prop:spinor-trace-formula}. If $D$ satisfies the Eichler condition,
  then the summation consists of only one term, and the spinor trace
  formula reduces back to the Vign\'eras-Voight formula.

  This paper is organized as follows. In
  Section~\ref{sec:optim-spin-select} we prove the optimal spinor
  selectivity theorem for an arbitrary genus of quaternion orders,
  thus generalizing the first two parts of
  Therorem~\ref{thm:Voight31}.  Section~\ref{sec:maclachlan} focuses
  on extending Maclachlan's relative conductor formula to all Eichler
  orders. In Section~\ref{sec:spinor-trace-formula}, we prove the
   spinor trace formula, which  generalizes the last part of
  Therorem~\ref{thm:Voight31}.

 % Let $\Pl(F)$ be the set of all  places of $F$, and
% $\Pl_\infty(F)$ (resp.~$\Pl_f(F)$) be the subset of infinite
% (resp.~finite) places. The set $\Pl_f(F)$ is naturally identified
% with the set of nonzero prime ideals of $O_F$ (i.e.~the finite
% primes of $F$).

\textbf{Notation.}  Throughout this paper, $\grp$ denotes a
finite prime of $F$.  If $M$
is a finite dimensional $F$-vector space or a finite $O_F$-module, then we
write $M_\grp$ for the $\grp$-adic completion of $M$.   Let
$\wbZ=\varprojlim \zmod{n}=\prod_p \Z_p$ be the profinite completion
of $\Z$. If $X$ is a finitely generated $\Z$-module or a finite
dimensional $\Q$-vector space, we set $\wh X:=X\otimes_\Z\wbZ$. For
example, $\whF$ is the ring of finite adeles of $F$, and
$\whO_F=\prod_\grp O_{F_\grp}$. Here $O_{F_\grp}$ denotes the ring of integers of
$F_\grp$, and the product runs over all finite
primes of $F$.  The reduced norm map of
$D$ is denoted as $\Nr: D\to F$. 
Given a set $Y\subseteq
\whD$, we write $Y^1$ for the subset of elements of reduced norm
$1$, that is,  $Y^1:=\{y\in Y\mid \Nr(y)=1\}$. In particular, $\whD^1=\ker(\whD^\times \to \whF^\times)$.
% \begin{equation}
%   \label{eq:116}

% \end{equation}

% completely multiplicative function on $\calI_F$.

% Given a commutative $F$-algebra $R$,
% we write $G(R)$ for the group of $R$-points of $G$. In particular,

% We shall make frequent use of the ``local-global'' dictionary of
% $O_F$-lattices \cite[Proposition~4.21]{curtis-reiner:1}. Let
% $V$ be a finite dimensional $F$-vector space, and $G$
% be a closed algebraic subgroup of $\GL(V)$. 
% We write $G(\whF)$ for the group of finite adelic points of $G$ \cite[\S5.1]{Platonov-Rapinchuk}. 
% Let $\Lambda\subset V$ be
% an $O_F$-lattice (i.e.~a finitely generated $O_F$-submodule
% that spans $V$ over $F$). For any $g=(g_\grp)\in G(\whF)$, there exists a unique
% $O_F$-lattice $\Lambda'\subset V$ such that
% $\Lambda'_\grp=g_\grp\Lambda_\grp$ for every  $\grp\in \Pl_f(F)$, so we
% put $g\Lambda:=\Lambda'$. 

\section{The optimal spinor selectivity theorem}  
\label{sec:optim-spin-select}

In this section, we derive the aforementioned theorem (Theorem~\ref{thm:selectivity}) of optimal spinor
selectivity that applies to any arbitrary genus of quaternion orders. Our new contribution to this already highly developed theory mostly  includes
verifying certain local condition for optimal embeddings in Lemma~\ref{lem:eichler-inv-minus1} at the places
of $F$
where the  Eichler invariant of the quaternion order is $-1$.  This
provides simplifications to the existing theory and makes
generalizations possible.

Different from Theorem~\ref{thm:Voight31}, the quaternion algebra $D$
in this section is allowed to be arbitrary (i.e.~not necessarily satisfying
the Eichler condition). By definition, $D$ is \emph{totally definite}
if and only if $F$ is a totally real field and
$D\otimes_{F, \sigma}\bbR$ is isomorphic to the Hamilton quaternion
algebra $\bbH$ for every embedding $\sigma: F\hookrightarrow \bbR$.
In the totally definite case, Theorem~\ref{thm:Voight31} does not
always hold without modifications, as shown by the following simple example.

\begin{ex}\label{ex:Dpinf}
  Let $p\in \bbN$ be a prime, and  $D_{p, \infty}$ be the unique quaternion
algebra over $\bbQ$ ramified precisely at  $p$ and
$\infty$. Let $\scrG_{p, \infty}$ be the genus of maximal $\Z$-orders
in $D_{p, \infty}$.  From \cite[Proposition~V.3.2]{vigneras}, if
$p\equiv 3\pmod{4}$, then there is a unique member
$[\calO_0]\in \Tp(\scrG_{p, \infty})$ such that
$\Emb(\Z[\sqrt{-1}], \calO_0)\neq \emptyset$, while
$\abs{\Tp(\scrG_{p, \infty})}$ tends to infinity as $p$ goes to
infinity.  This already violates part (ii) of Theorem~\ref{thm:Voight31}.
\end{ex}

As remarked by M.~Arenas et al.~\cite[Remark,
  p.~99]{M.Arenas-et.al-opt-embed-trees-JNT2018}, one needs the notion
  of \emph{spinor genus} to extend the optimal selectivity theorem to
  the totally definite case. This notion was previous studied by
  Brzezinski 
  \cite[\S1]{Brzezinski-Spinor-Class-gp-1983}. 

\begin{defn}\label{defn:spinor-genus}
  (1) Let $\scrG$ be an arbitrary genus of orders in $D$.  Two orders
  $\calO, \calO'\in \scrG$ are said to be in the same \emph{spinor
    genus} (and denoted by $\calO\sim \calO'$) if there exists
  $x\in D^\times\whD^1$ such that $\wcO'=x\wcO x^{-1}$. 

(2) The  \emph{spinor genus} of $\calO$ is the set $[\calO]_\sg$ consisting
  of all orders $\calO'$ with $\calO'\sim \calO$. The set of spinor
  genera within $\scrG$ is denoted by $\SG(\scrG)$, that is, 
  $\SG(\scrG):=\{[\calO]_\sg\mid \calO\in \scrG\}$. Often we write
  $\SG(\calO)$ for $\SG(\scrG)$ and regard it as a pointed set
  with base point $[\calO]_\sg$. 
 
 (3) Given orders
  $B\subseteq K$ and $\calO\in \scrG$, we define the \emph{optimal
    spinor selectivity symbol} as follows
\begin{equation}
  \label{eq:6}
  \Delta(B, \calO)=
  \begin{cases}
    1 \qquad &\text{if } \exists\, \calO'\in [\calO]_\sg\text{ such that }\Emb(B, \calO')\neq \emptyset, \\
    0 \qquad &\text{otherwise}.
  \end{cases}
  \end{equation}
\end{defn}

% Clearly, if two orders $\calO$ and $\calO'$ are in the same spinor
% genus, then they certainly belong to the same genus $\scrG$. Thus this
% defines a classification of orders in $\scrG$ that is coarser than
% ``being in the same type''. Nevertheless, the two concept coincide
% when $D$ satisfies the Eichler condition. 

\begin{rem}\label{rem:Eichler-con}
If  $D$ satisfies the Eichler condition, Brzezinski
\cite[Proposition~1.1]{Brzezinski-Spinor-Class-gp-1983} shows that each spinor genus
consists of exactly one type. Thus in this case, $\Delta(B, \calO)=1$
if and only if $\Emb(B, \calO)\neq \emptyset$. 
\end{rem}

\begin{defn}\label{defn:oss}
  We say $B$ is \emph{optimally spinor selective} (selective for
  short) for  
  the genus $\scrG$ if $\Delta(B, \calO')=0$ for some but not all $[\calO']_\sg\in \SG(\scrG)$. If $B$ is selective for $\scrG$, then a
  spinor genus $[\calO]_\sg$ with $\Delta(B, \calO)=1$
  is said to be \emph{selected} by $B$.
\end{defn}

To give a preliminary characterization  of the  spinor genera  selected by
$B$, we take the ``two class field'' approach as mapped out by
Arenas-Carmona \cite[\S3]{Arenas-Carmona-Spinor-CField-2003} and
combine it with inputs from Voight's work \cite[\S31]{voight-quat-book}. 

First, let us introduce the class field $\Sigma_\scrG/F$ attached to
the genus $\scrG$ as mentioned in the introduction. 
 Following
\cite[\S III.4]{vigneras}, we write $F_D^\times$ for the subgroup
of $F^\times$ consisting of the elements that are positive at each
infinite place of $F$ ramified in $D$.  
Let $\Nr: D^\times\to F^\times$ be the reduced norm map. 
The Hasse-Schilling-Maass
theorem \cite[Theorem~33.15]{reiner:mo}
\cite[Theorem~III.4.1]{vigneras} implies that
$  \Nr(D^\times)=F_D^\times$.
Moreover, we have $\Nr(\whD^\times)=\whF^\times$, the finite idele
group of $F$.   Let  $\calN(\wcO)$ be the normalizer of $\wcO$
 in $\whD^\times$. There is an adelic description of $\SG(\calO)$  as follows
\cite[Propositions~1.2 and 1.8]{Brzezinski-Spinor-Class-gp-1983}
\begin{equation}
  \label{eq:118}
\SG(\calO)\simeq  (\grD^\times\whD^1)\bsh \whD^\times/\calN(\wcO)\xrightarrow[\simeq]{\Nr}
  F_\grD^\times\bsh \whF^\times/\Nr(\calN(\wcO)), 
\end{equation}
where the  two double coset spaces are canonically bijective via
the reduced norm map.  Clearly, the
group $\Nr(\calN(\wcO))$ depends only on the genus $\scrG$ and not on
the choice of $\calO$.

% Once again, the group 
% $F_\grD^\times\bsh \whF^\times/\Nr(\calN(\wcO))$ in (\ref{eq:118})
% depends only on $\scrG$, so it will be called the \emph{group of spinor genera in $\scrG$} and denoted by
% $\grG_\sg(\scrG)$.  Since $\Nr(\calN(\wcO))$ is an open subgroup of
% $\whF^\times$ containing $(\whF^\times)^2$,  the group $\grG_\sg(\scrG)$ is a
% \emph{finite} elementary $2$-group
% \cite[Proposition~3.5]{Linowitz-Selectivity-JNT2012}. 

\begin{defn}[{\cite[\S2]{Arenas-Carmona-Spinor-CField-2003},
    \cite[\S3]{Linowitz-Selectivity-JNT2012}}] \label{defn:spinor-field}
  The \emph{spinor genus field}\footnote{This field is often
    called the \emph{spinor class field} in the literature
    \cite{M.Arenas-et.al-opt-embed-trees-JNT2018, Arenas-Carmona-Spinor-CField-2003, Arenas-Carmona-2013,
      Arenas-Carmona-cyclic-orders-2012}. However,  we are going to
    introduce a concept called \emph{spinor class} following
    \cite[\S1]{Brzezinski-Spinor-Class-gp-1983}.  To avoid confusion,  the field
    $\Sigma_\scrG$ will be called the spinor genus field since it is
    uniquely determined by the genus $\scrG$. Moreover, if $F$ is a
    quadratic field, and $\calO$ is an Eichler order, then
    $\Sigma_\scrG$ is a subfield of the classical \emph{(strict) genus
      field} in \cite[Definition~15.29]{Cohn-invitation-Class-Field}
    or \cite[\S6]{Cox-Primes}, so the terminology is consistent in
    that sense. } of $\scrG$ is the abelian field extension
  $\Sigma_\scrG/F$ corresponding to the open subgroup
  $F_D^\times\Nr(\calN(\wcO))\subseteq \whF^\times$ via the class
  field theory \cite[Theorem~X.5]{Lang-ANT}. 
\end{defn}
Since $\Nr(\calN(\wcO))$ is an open subgroup of
$\whF^\times$ containing $(\whF^\times)^2$, the Galois group $\Gal(\Sigma_\scrG/F)$ is a
finite elementary $2$-group
\cite[Proposition~3.5]{Linowitz-Selectivity-JNT2012}.
We have a canonical identification of the pointed sets
\begin{equation}
  \label{eq:8}
\SG(\calO)\simeq  F_D^\times\bsh \whF^\times/
  \Nr(\calN(\wcO))\simeq \Gal(\Sigma_\scrG/F), 
\end{equation}
where the base point $[\calO]_\sg$ is identified with the identity
element of $\Gal(\Sigma_\scrG/F)$. This equips $\SG(\calO)$ with an abelian group structure. 
Given another order $\calO'\in \scrG$, we define $\rho(\calO, \calO')$
to be the element of $\Gal(\Sigma_\scrG/F)$ identified with
$[\calO']_\sg\in \SG(\calO)$  via (\ref{eq:8}). 
  More canonically, we regard the base
point free set $\SG(\scrG)$ as a principal homogeneous space over
$\Gal(\Sigma_\scrG/F)$ via (\ref{eq:8}). Then $\rho(\calO, \calO')$ is
the unique element of $\Gal(\Sigma_\scrG/F)$ that sends $[\calO]_\sg$
to $[\calO']_\sg$.  Since $\Gal(\Sigma_\scrG/F)$ is an elementary
$2$-group,  $\rho$ defines a map $\scrG\times
\scrG\to \Gal(\Sigma_\scrG/F)$ that is symmetric in its two
variables. By definition, $\rho(\calO, \calO')$ depends only on the spinor genera of $\calO$ and $\calO'$.   Moreover, $\rho(\calO, \calO')=1$ if and only if
$\calO\sim \calO'$. If we write the group law of
$\Gal(\Sigma_\scrG/F)$ additively, then  
\begin{equation}\label{eq:26}
\rho(\calO, \calO'')=\rho(\calO, \calO')+\rho(\calO', \calO''), \quad
\forall \calO, \calO', \calO''\in \scrG.   
\end{equation}

\begin{ex}\label{ex:link-ideal}
Suppose for the moment that
$\Nr(\calO_\grp^\times)=O_{F_\grp}^\times$ for every prime
$\grp\subset O_F$ so that $\Sigma_\scrG/F$ is unramified at all the finite
places.   For instance, this assumption holds for all Eichler
orders in $D$. 
  Let $\calO$ and $\calO'$ be two members of
  $\scrG$. There exists an $O_F$-lattice $I$ in $\grD$ linking $\calO$
  and $\calO'$ as in \cite[\S I.4]{vigneras}, that is,
  $\calO'=\calO_l(I)$ and $\calO=\calO_r(I)$, where 
  \begin{equation}
    \label{eq:12}
   \calO_l(I):=\{\alpha\in D\mid \alpha I\subseteq I\}, \quad
   \calO_r(I):=\{\alpha\in D\mid  I \alpha\subseteq I\}. 
  \end{equation}
Such a lattice $I$ is 
  locally principal both as a fractional right $\calO$-ideal and   as a
  fractional left $\calO'$-ideal.  Then
  $\rho(\calO, \calO')\in \Gal(\Sigma_\scrG/F)$ is given by the Artin symbol
  $(\Nr(I), \Sigma_\scrG/F)$. See
  \cite[\S2]{M.Arenas-et.al-opt-embed-trees-JNT2018} and
  \cite[\S31.1.9]{voight-quat-book}. 
\end{ex}

% ,  Since $\Gal(\Sigma_\scrG/F)$ is an elementary
% $2$-group, $\rho: \scrG\times \scrG\to $\Gal(\Sigma_\scrG/F)$ is 

% \begin{enumerate}[label=(\alph*)]
% \item $\rho(\calO, \calO')=1$ 
% \item $\rho(\calO, \calO')=\rho(\calO', \calO)$;
% \item   
% \end{enumerate}

Next, we introduce another class field $E_\op/F$ closely related to the optimal
embeddings. 
For each $F$-embedding
$\varphi: K\hookrightarrow \grD$, consider the 
following sets
\begin{align}
  \label{eq:25}
\calE(\varphi, B, \calO)&:=\{\alpha\in \grD^\times\mid
  \varphi(K)\cap \alpha\calO \alpha^{-1}=\varphi(B)\},\\
  \calE_\grp(\varphi, B, \calO)&:=\{g_\grp\in \grD_\grp^\times \mid \varphi(K_\grp)\cap g_\grp\calO_\grp
  g_\grp^{-1}=\varphi(B_\grp)\}, \quad \forall \text{ prime }
                                 \grp\subset O_F, \label{eq:159}\\ 
  \label{eq:158}  \whE(\varphi, B, \calO)&:=\{g=(g_\grp)\in \whD^\times \mid \varphi(K)\cap g\wcO
  g^{-1}=\varphi(B)\}. 
\end{align}
 Given any other $F$-embedding
$\varphi':K\hookrightarrow \grD$ and another order $\calO'\in \scrG$, we
pick $\alpha\in \grD^\times$ and $x\in \whD^\times$ such that
$\varphi'=\alpha^{-1} \varphi \alpha$ and $\wcO'=x\wcO x^{-1}$. Then 
\begin{equation}
  \label{eq:170}
\whE(\varphi', B, \calO')=\alpha^{-1}\whE(\varphi, B, \calO)x^{-1}. 
\end{equation}
% When  $\varphi$ is fixed and clear from
%  the context,  we often drop it from the notation and simply write  $\calE(B,
%  \calO)$, $\whE(B, \calO)$ and $\calE_\grp(B, \calO)$ instead. 

Since condition (\ref{eq:3}) is assumed  to hold throughout this
paper, there is always an order $\calO\in \scrG$ such that $\Emb(B,
\calO)\neq \emptyset$. Keep such an $\calO$ and an optimal
embedding $\varphi\in \Emb(B, \calO)$ fixed. 
 For simplicity, we identify $K$
with its image $\varphi(K)\subset \grD$. Clearly,
$\whE(\varphi, B, \calO)$ is left translation invariant by
$\whK^\times$ and right translation invariant by
$\calN(\wcO)$. Moreover, it contains $1$ by the choice of $\calO$ and $\varphi$,  so we
have 
\begin{equation}
  \label{eq:167}
  \whE:=\whE(\varphi, B, \calO)\supseteq \whK^\times\calN(\wcO). 
\end{equation}

\begin{defn}[{\cite[\S2]{Arenas-Carmona-Spinor-CField-2003},
    \cite[\S3]{Linowitz-Selectivity-JNT2012}}] \label{defn:op-rep-field}
  The \emph{optimal representation field} attached to the pair $(B, \scrG)$ is the abelian field extension
  $E_\op/F$ corresponding to $F_D^\times\Nr(\whE)\subseteq \whF^\times$ via the class
  field theory.  
\end{defn}

A priori, for this definition to make sense, one needs to know that
$F_D^\times\Nr(\whE)$ is an open subgroup of
$\whF^\times$. This has already been verified by Voight
\cite[\S 31.3.14]{voight-quat-book}, who also shows  that $[\whF^\times: F_\grD^\times\Nr(\whE)]\leq 2$.  
His proof
makes use of the
following chain of inclusions 
\begin{equation}
  \label{eq:173}
  F_K^\times\Nr(\whK^\times)\subseteq   F_\grD^\times
  \Nr(\whK^\times)\Nr(\calN(\wcO))\subseteq
  F_\grD^\times\Nr(\whE)\subseteq \whF^\times.
\end{equation}
Here  $F_K^\times$ is the subgroup
of $F^\times$ consisting of the elements
that are positive at each infinite place of $F$ that is ramified in
$K/F$. The assumption that $K$ is $F$-embeddable into $\grD$ implies
that $ F_K^\times\subseteq F_\grD^\times$.   Observe that $F_K^\times\Nr(\whK^\times)$ is an open subgroup of $\whF^\times$ of index $2$ by the class field theory. In fact, this is the key idea behind Voight's arguments.  

%From class field theory, we have $[\whF^\times: F_K^\times\Nr(\whK^\times)]=2$. This is actually the key point behind Voight's arguments.  

We claim that
the 
group $F_D^\times\Nr(\whE)$ is uniquely determined by the pair
$(B, \scrG)$.  In other words, if $\calO'\in \scrG$ is another order
with $\varphi'\in \Emb(B, \calO')$, then
$F_D^\times\Nr(\whE')= F_\grD^\times\Nr(\whE)$, where 
$\whE':=\whE(\varphi', B, \calO')$.  Indeed,  from (\ref{eq:170}), we
have  $F_D^\times\Nr(\whE')=
F_\grD^\times\Nr(\whE)\Nr(\alpha x)^{-1}$ for suitable $\alpha\in
D^\times$ and $x\in \whD^\times$.  On the other hand,  
 $\alpha
x\in \whE$ by definition  (\ref{eq:158}). The claim is verified 
since $F_\grD^\times\Nr(\whE)$ is
a subgroup of $\whF^\times$.

At each finite prime
$\grp$ of $F$,  put $\calE_\grp:=\calE_\grp(\varphi, B, \calO)$  for
simplicity. There is  a similar chain of
inclusions as the one in (\ref{eq:173}): 
\begin{equation}
  \label{eq:155}
\Nr(K_\grp^\times)\subseteq
\Nr(K_\grp^\times)\Nr(\calN(\calO_\grp))\subseteq
\Nr(\calE_\grp)\subseteq F_\grp^\times.
\end{equation}
The same argument as those in \cite[\S31.3.14]{voight-quat-book} shows
that $\Nr(\calE_\grp)$ is a subgroup of $F_\grp^\times$ of index at
most $2$. Moreover, $\Nr(\calE_\grp)$ depends only on the isomorphism classes of the local orders $B_\grp$ and
$\calO_\grp$ by a local version of the calculation above.

% Without lose of generality, we assume $m_\grp(B)\neq 0$ for every
%  $\grp\in \Pl_f(F)$ so that there exists $\calO\in \scrG$ with
% $\Emb(B, \calO)\neq \emptyset$. Fix such .

% We cite the following lemma of Voight . 
% \begin{lem}\label{lem:sel-sandwich}
%  The set  
% \end{lem}
%  Similar to $F_\grD^\times$, we define
% The proof of  hinges upon 
% See \cite[\S
% 31.3.14]{voight-quat-book} for the details.  

Note that the first two terms of (\ref{eq:173})
 correspond via the class field theory to $K$ and $K\cap \Sigma_\scrG$
 respectively. It follows that 
 \begin{equation}
     \label{eq:5}
     K\supseteq K\cap \Sigma_\scrG \supseteq E_\op \supseteq F.
 \end{equation}
 
 % \begin{equation}
 %   \label{eq:5}
 %   F\subseteq E_\op\subseteq K\cap \Sigma_\scrG\subseteq K. 
 % \end{equation}

The following lemma is adapted from
\cite[Lemma~2.2]{M.Arenas-et.al-opt-embed-trees-JNT2018} (see also
\cite[Theorem~3.2]{Maclachlan-selectivity-JNT2008} and 
\cite[Proposition~31.4.4]{voight-quat-book}), so we omit its proof.  
 \begin{lem}\label{lem:prim-criterion}
   Let $\calO\in \scrG$ be an order such that there exists $\varphi\in
   \Emb(B, \calO)$. Given another order $\calO'\in \scrG$, we have
   $\Delta(B, \calO')=1$ if and only if $\rho(\calO, \calO')\in \Gal(\Sigma_\scrG/F)$
   restricts to identity on $E_\op$. 
 \end{lem}

Since $[K: F]=2$, the field $E_\op$ coincides with either $F$ or $K$
by (\ref{eq:5}). In light of the 
  the identification $\SG(\calO)\simeq  \Gal(\Sigma_\scrG/F)$ in
  (\ref{eq:8}), the corollary below follows directly from the above lemma. 

\begin{cor}\label{cor:2case-Eop}
The order   $B\subset K$ is selective for the genus $\scrG$ if and only if
$E_\op=K$. More explicitly, 
\begin{enumerate}[label=(\roman*)]
\item if $E_\op=F$, then $\Delta(B, \calO')=1$ for every spinor genus
  $[\calO']_\sg\in \SG(\scrG)$, and $B$ is non-selective for the genus
  $\scrG$;
\item if $E_\op=K$, then $K\subseteq \Sigma_\scrG$, and exactly half
  of the spinor genera in $\SG(\scrG)$ are selected by $B$.  
\end{enumerate}
In particular, if $K\cap \Sigma_\scrG=F$, then $B$ is non-selective
for the genus $\scrG$.  
\end{cor}
So far everything has been pretty standard: almost all the ideas have
appeared in the literature somewhere. Now we employ the local-global
compatibility of class field theory
\cite[\S6]{Tate-global-CFT-CasselsFrohlich} to reduce it to purely
local considerations.  This will eventually allow us to insert our
own input into the theory.

\begin{lem}\label{lem:K-in-Sigma}
We have  $K\subseteq \Sigma_\scrG$  if and only if both of the
following conditions hold: 
\begin{enumerate}
\item[(i)] $F_K^\times=F_\grD^\times$, or equivalently by weak approximation,  $K$ and $\grD$ are ramified at
exactly the same (possibly empty) set of real places of $F$;
\item[(ii)]  $\Nr(\calN(\calO_\grp))\subseteq \Nr(K_\grp^\times)$ 
 for every finite prime $\grp$ of $F$. 
\end{enumerate}
\end{lem}
\begin{proof}
% If $F_K^\times=F_\grD^\times$, then by weak
% approximation, $K$ and $\grD$ are necessarily ramified at
% exactly the same set of real places of $F$; the converse follows
% directly from definition. This verifies the equivalence conditions in
% (i). 

% If $K$ and $\grD$ are ramified at
% exactly the same set of real places of $F$, then
% $F_\grD^\times=F_K^\times$ by definition; conversely, by weak
% approximation,  the equality $F_K^\times=F_\grD^\times$ implies
% that $K$ and $\grD$ are ramified at
% exactly the same set of real places of $F$. 

Recall that $F_K^\times\subseteq F_D^\times$ since $K$ embeds into $D$
over $F$ by assumption.   
Clearly, (i) and (ii) guarantee that $F_K^\times
  \Nr(\whK^\times)\supseteq F_\grD^\times
\Nr(\calN(\wcO))$, and hence
 $K\subseteq \Sigma_\scrG$. 

Conversely, suppose that $K\subseteq \Sigma_\scrG$ so that $F_K^\times
  \Nr(\whK^\times)\supseteq F_\grD^\times\Nr(\calN(\wcO))$.
%the first
%inclusion in (\ref{eq:173}) becomes an equality. 
By definition,
$\Sigma_\scrG$ splits completely at all the real places of $F$
that are unramified in $\grD$, and hence (i) necessarily holds.   It follows from $F_K^\times=F_\grD^\times$ that 
\begin{equation}
  \label{eq:177}
 \Nr(\calN(\wcO))\subseteq F_K^\times
  \Nr(\whK^\times). 
\end{equation}
If $\grp$ splits in $K$, then
$\Nr(K_\grp^\times)=F_\grp^\times$, so
$\Nr(\calN(\calO_\grp))\subseteq \Nr(K_\grp^\times)$ automatically. 
  On the other hand, according to the local-global compatibility of
  class field theory \cite[\S6]{Tate-global-CFT-CasselsFrohlich}, for
  each $\grp$ non-split in $K$, 
  there is a commutative diagram 
\[\begin{tikzcd}
  F_\grp^\times/\Nr(K_\grp^\times) \ar[d]\ar[r, "\simeq"] &  \Gal(K_\grp/F_\grp)\ar[d, equal]\\
  \whF^\times/F_K^\times\Nr(\whK^\times) \ar[r, "\simeq"] &    \Gal(K/F).
\end{tikzcd}\]
Given $a_\grp\in F_\grp^\times$, we have $a_\grp\in
\Nr(K_\grp^\times)$ if and only if   $a:=(\ldots, 1, a_\grp, 1, \ldots)$ lies in $F_K^\times\Nr(\whK^\times)$. Therefore, condition (ii) is necessary  as well. 
\end{proof}

According to Corollary~\ref{cor:2case-Eop}, for $B$ to be selective for
the genus $\scrG$, it is necessary that $K\subseteq
\Sigma_\scrG$. Suppose that this is the case for
the moment. Then $F_K^\times=F_\grD^\times$ by this assumption. The same proof as Lemma~\ref{lem:K-in-Sigma} shows that 
\begin{equation}\label{eq:262}
%  \label{eq:187}
E_\op= K \quad \text{if and
  only if}\quad    \Nr(K_\grp^\times)=
\Nr(\calE_\grp), \forall \text{ finite prime } \grp\subset O_F. 
\end{equation}
Note that $\Nr(K_\grp^\times)=
\Nr(\calE_\grp)$ holds automatically at every prime $\grp$ split in
$K$ since $\Nr(K_\grp^\times)
=F_\grp^\times$ for such $\grp$. At each remaining  $\grp$ we already have $
\Nr(K_\grp^\times)=\Nr(K_\grp^\times)\Nr(\calN(\calO_\grp))$ thanks to
the assumption $K\subseteq \Sigma_\scrG$. 
 In light of the chain of inclusions (\ref{eq:155}), 
one further asks whether the following equality holds:
\begin{equation}
  \label{eq:257}
 \Nr(K_\grp^\times)\Nr(\calN(\calO_\grp))=\Nr(\calE_\grp).  
\end{equation}
It is often convenient to treat this  purely as a
local question, discussed separately from the assumption $K\subseteq
\Sigma_\scrG$.  
\begin{ex}
  Let $\grp$ be a finite prime of $F$ that is unramified
in $K$ and coprime to the reduced discriminant $\grd(\calO)$ of $\calO$. The
latter condition implies that 
 $\calO_\grp\simeq M_2(O_{F_\grp})$, the ring of $2\times 2$ matrices
 with entries in $O_{F_\grp}$. According to
 \cite[Theorem~II.3.2]{vigneras}, all optimal embeddings of
 $B_\grp$ into $\calO_\grp$ are $\calO_\grp^\times$-conjugate, so
 $\calE_\grp=K_\grp^\times\calO_\grp^\times$. On the other hand, 
$\Nr(K_\grp^\times)\supseteq O_{F_\grp}^\times$, and 
$\calN(\calO_\grp)=F_\grp^\times\calO_\grp^\times$. Therefore, 
\begin{equation}
  \label{eq:169}
\Nr(K_\grp^\times)=
\Nr(K_\grp^\times)\Nr(\calN(\calO_\grp))=
\Nr(\calE_\grp).
\end{equation}
This already shows that (\ref{eq:257}) holds for almost all finite
primes of $F$. 
\end{ex}
A key to Voight's optimal selectivity theorem as
quoted in Theorem~\ref{thm:Voight31} is the following equality result by himself 
 \cite[Proposition~31.5.7]{voight-quat-book}. 

 \begin{prop}\label{prop:voight-eichler}
   If $\calO_\grp$ is an Eichler order, then  equality (\ref{eq:257}) holds  at
$\grp$. 
 \end{prop}

To go beyond the Eichler orders, we  recall the notion of  Eichler invariants from
\cite[Definition~1.8]{Brzezinski-1983}.

\begin{defn}\label{defn:eichler-invariant}
  Let $\grp$ be a finite prime of $F$,  $\grk_\grp:= O_F/\grp$
be the finite residue field of $\grp$,  and $\grk_\grp'/\grk_\grp$ be the unique
quadratic field extension.  When $\calO_\grp\not\simeq
M_2(O_{F_\grp})$, the quotient of $\calO_\grp$ by its Jacobson radical
$\grJ(\calO_\grp)$ falls into the following three cases: 
\[\calO_\grp/\grJ(\calO_\grp)\simeq \grk_\grp\times \grk_\grp, \qquad \grk_\grp,
\quad\text{or}\quad \grk_\grp', \]
and the \emph{Eichler invariant} $e_\grp(\calO)$ of $\calO$ at $\grp$ is defined to be
$1, 0, -1$ accordingly.  As a convention, if $\calO_\grp\simeq
M_2(O_{F_\grp})$, then $e_\grp(\calO)$ is defined to be
$2$.

Similarly, we write $(K/\grp)$ for the symbol\footnote{This is usually
called the Artin symbol, but we want to distinguish it from the Artin
symbol $(\gra, K/F)\in \Gal(K/F)$ to be used later where we identify
$\Gal(K/F)$ with $\zmod{2}$.} that 
takes value $1, 0, -1$ according to whether $\grp$ is split, ramified or
inert in $K/F$. 
\end{defn}
 For
example, if $D$ is ramified at $\grp$ and $\calO_\grp$ is maximal, then
$e_\grp(\calO)=-1$. It is shown in
\cite[Proposition~2.1]{Brzezinski-1983} that 
$e_\grp(\calO)=1$ if and only if
$\calO_\grp$ is a non-maximal Eichler order (particularly,
$D$ is split at $\grp$).  As a result, if $\calO$ is an Eichler
order, then $e_\grp(\calO)\neq 0$ for every finite prime $\grp$.

% Return to the the assumption $\calO$ being arbitrary.  For each
% finite prime 
% $\grp$, we pick $\varphi_\grp\in \Emb(B_\grp, \calO_\grp)$ and define
% \begin{equation}
% \label{eq:159}
%   \calE_\grp(\varphi_\grp, B, \calO):=\{g_\grp\in \grD_\grp^\times \mid \varphi_\grp(K_\grp)\cap g_\grp\calO_\grp
%   g_\grp^{-1}=\varphi(B_\grp)\}. 
% \end{equation}
% For simplicity, put $\calE_\grp:=\calE_\grp(\varphi_\grp, B,
% \calO)$.  It will be  shown in Lemma~\ref{} (cf.~\cite[\S31.3.34]{voight-quat-book})
% that $\Nr(\calE_\grp)$ is  a subgroup of $F_\grp^\times$ of index at
%   most $2$ containing $\Nm_{K/F}(K_\grp^\times)$. Moreover, 
%    $\Nr(\calE_\grp)$ depends only on $B_\grp$ and $[\calO_\grp]$,
% that is,  it  depends neither on the choice of $\calO\in \scrG$ nor
% on the choice of $\varphi_\grp\in \Emb(B_\grp, \calO_\grp)$.

% $\grp\not\in S$ by Remark~\ref{rem:local-cond-eichler-orders}, so  $   \Nr(K_\grp^\times)=
% \Nr(\calE_\grp)$ at every $\grp\not\in S$.  Thus 
% (\ref{eq:163}) is indeed a necessary and sufficient condition for $B$
% to be selective. 

% In the next lemma, we work purely in the local case.

We extend Voight's equality result above to local quaternion orders with Eichler invariant
$-1$. 
  % Let $\grp$ be a finite prime with  $e_\grp(\calO)=-1$. Then
  % $\Nr(K_\grp^\times)\Nr(\calN(\calO_\grp))= \Nr(\calE_\grp)$ for all
  % $B_\grp\subseteq O_{K_\grp}$ optimally embeddable into $\calO_\grp$.

\begin{lem}\label{lem:eichler-inv-minus1}
   If $e_\grp(\calO)=-1$, then $\Nr(K_\grp^\times)\Nr(\calN(\calO_\grp))= \Nr(\calE_\grp)$.
\end{lem}
% Indeed, if
% $e_\grp(\calO)=1$ or $2$, then $\calO_\grp$ is an Eichler order, and the equality
% (\ref{eq:257}) is a result of
% Voight \cite[Proposition~31.5.7]{voight-quat-book}. Thus we focus on the case
% $e_\grp(\calO)=-1$. 

\begin{proof}
  Without lose of generality, we assume that $\grp$ is non-split in
  $K$.  By
  \cite[Proposition~3.1]{Brzezinski-1983}, $\calO_\grp$ is a Bass
  order (see \cite[\S37]{curtis-reiner:1},
  \cite[\S1]{Brzezinski-crelle-1990}, \cite{Voight-basic-orders}
  and \cite[\S3.1]{xue-yu-zheng:spIII}). Let $L_\grp/F_\grp$ be the
  unique unramified quadratic field extension.  It is shown in
  \cite[Proposition~1.12]{Brzezinski-crelle-1990} that
  $\Emb(O_{L_\grp}, \calO_\grp)\neq \emptyset$, which implies that
  $\Nr(\calO_\grp^\times)=O_{F_\grp}^\times$ since
  \begin{equation}
    \label{eq:263}
O_{F_\grp}^\times\supseteq \Nr(\calO_\grp^\times)\supseteq
  \Nm_{L_\grp/F_\grp}(O_{L_\grp}^\times)=O_{F_\grp}^\times.      
  \end{equation}
Thus if $K_\grp/F_\grp$ is ramified, then
$\Nr(K_\grp^\times)\Nr(\calN(\calO_\grp))=F_\grp^\times$ and the
equality is trivial again.

Let $\grd(\calO_\grp)$ be the reduced discriminant of $\calO_\grp$,
and $\nu_\grp: F_\grp^\times \twoheadrightarrow \bbZ$ be the
normalized discrete valuation of $F_\grp$.  From \cite[Corollary~3.2]{Brzezinski-1983}, $\nu_\grp(\grd(\calO_\grp))$
is odd if and only if $\grD_\grp$ is division. When
$\grD_\grp$ is division, there exists $u\in
\calN(\calO_\grp)$ such that $\nu_\grp(\Nr(u))$ is odd by
\cite[Theorem~2.2]{Brzezinski-crelle-1990}. It follows that $\Nr(\calN(\calO_\grp))=F_\grp^\times$ in this case and equality is
trivial once more.

Lastly, assume that $K_\grp=L_\grp$ and $\grD_\grp\simeq M_2(F_\grp)$. In this
case, $\Nr(\calN(\calO_\grp))=\Nr(K_\grp^\times)=F_\grp^{\times2}O_{F_\grp}^\times$ by
\cite[Theorem~2.2]{Brzezinski-crelle-1990}. If $B_\grp=O_{K_\grp}$,
then $\calN(\calO_\grp)$ acts transitively on
$\Emb(B_\grp, \calO_\grp)$ as shown in the start of the proof of
\cite[Theorem~3.3, p.~178]{Brzezinski-crelle-1990}. In other words,
$\calE_\grp=K_\grp^\times\calN(\calO_\grp)$, and the equality
(\ref{eq:257}) holds. Now suppose that $B_\grp$ is non-maximal, and
$B_\grp'\supseteq B_\grp$ is the unique $O_{F_\grp}$-order in $K_\grp$
such that  $B_\grp=O_{F_\grp}+\grp B_\grp'$. Let $\calO_\grp'\supset \calO_\grp$
be the unique minimal overorder of $\calO_\grp$.  We claim that
$\Nr(\calN(\calO_\grp'))=\Nr(K_\grp^\times)$ holds for $\calO_\grp'$
as well. Indeed, if
$\calO_\grp'\not\simeq M_2(O_{F_\grp})$, then it has Eichler invariant
$-1$ again by \cite[Corollary~3.2]{Brzezinski-1983}, so
the equality holds 
as shown above. 
If 
$\calO_\grp'\simeq M_2(O_{F_\grp})$, then the equality  holds by (\ref{eq:169}). 
The claim is verified.  Fix an optimal embedding
$\varphi_\grp: B_\grp\to \calO_\grp$, and put
$\calE_\grp':=\calE_\grp(\varphi_\grp, B_\grp', \calO_\grp')$. 
According to
\cite[Lemma~3.7]{Brzezinski-crelle-1990}, every optimal embedding
$B_\grp\to \calO_\grp$ extends to an optimal embedding
$B_\grp'\to \calO_\grp'$, which implies that $\calE_\grp \subseteq
\calE_\grp'$. Thus to show that $\Nr(K_\grp^\times)=\Nr(\calE_\grp)$,
it is enough to show that
$\Nr(K_\grp^\times)=\Nr(\calE_\grp')$ for the pair $(B_\grp', \calO_\grp')$. Iterating the above argument
using the ascending chain of orders in
\cite[Corollary~3.2]{Brzezinski-1983}, we eventually arrive at a pair
$(B_\grp'', \calO_\grp'')$ where either $B_\grp''=O_{K_\grp}$ or
$\calO_\grp''\simeq M_2(O_{F_\grp})$. It has already been shown that
$\Nr(K_\grp^\times)=\Nr(\calE_\grp'')$, and the lemma is proved.  
\end{proof}

% Clearly, $\calE_\grp$ contains $1$ by the choice of
% $\varphi_\grp$, and it
% is invariant
% under the left multiplication of $\varphi_\grp(K_\grp^\times)$. 
% From
% the local class field theory, $\Nm_{K/F}(K_\grp^\times)$ is a subgroup of
% $F_\grp^\times$ of index at most $2$. Since 
% $\Nr(\calE_\grp)$ is 
% translation invariant under $\Nm_{K/F}(K_\grp^\times)$, it is a union of
% cosets of the latter group. It follows that $\Nr(\calE_\grp)$ is 
% a subgroup of $F_\grp^\times$ of index at most $2$
% (cf.~\cite[\S31.3.34]{voight-quat-book}) containing $\Nm_{K/F}(K_\grp^\times)$. 

In the next theorem, we no longer keep $\calO\in \scrG$ fixed with
$\Emb(B, \calO)\neq \emptyset$. Rather, $\calO$ is allows to be any
arbitrary member of $\scrG$. This does not affect the subgroup
$\Nr(\calE_\grp)\subseteq F_\grp^\times$, which depends only on $B$
and the genus $\scrG$. Since $e_\grp(\calO)$ is independent of the
choice of $\calO\in \scrG$ as well, it makes sense to put
$e_\grp(\scrG):=e_\grp(\calO)$ for any $\calO\in \scrG$.

\begin{thm}\label{thm:selectivity}
  Let $\scrG$ be an arbitrary  genus of orders in $D$. Let  $K/F$ be a
  quadratic field extension that is $F$-embeddable into $\grD$, and
  $B$ be an order in $K$.
 Suppose that condition (\ref{eq:3}) holds for $B$ and $\scrG$. Let $S$ be the following finite (possibly
  empty) set of primes of $F$:
  \begin{equation}
    \label{eq:178}
    S:=\left\{\grp\, \middle\vert\, (K/\grp)\neq
      1\quad \text{and}\quad e_\grp(\scrG)=0    \right\}, 
  \end{equation}    
Then $B$ is (optimally spinor) selective for $\scrG$ if and only if
    \begin{equation}
      \label{eq:163}
      K\subseteq \Sigma_\scrG,\quad \text{and} \quad  \Nm_{K/F}(K_\grp^\times)=
\Nr(\calE_\grp)\text{ for every } \grp\in S. 
    \end{equation}
    If $B$ is selective, then
    \begin{enumerate}
    \item  for any two orders   $\calO, \calO'\in \scrG$,  
   \begin{equation}
\label{eq:164}
    \Delta(B, \calO)=\rho(\calO, \calO')|_K
+\Delta(B, \calO'), 
  \end{equation}
  where $\rho(\calO, \calO')|_K$ is the restriction of $\rho(\calO,
  \calO')\in \Gal(\Sigma_\scrG/F)$ to $K$,  and the
  summation on the right is taken inside $\zmod{2}$ with the canonical
  identification $\Gal(K/F)\simeq \zmod{2}$;

  \item    exactly half of the spinor genera in $\SG(\scrG)$ are 
    selected by $B$. 
    \end{enumerate}
\end{thm}
%     Moreover, part (b) of Theorem~\ref{thm:eichler-indef} remains true in this
% general setting as well.

\begin{proof}
 The first part of the
theorem follows from combining Corollary~\ref{cor:2case-Eop} with
(\ref{eq:262}),  Proposition~\ref{prop:voight-eichler} and
Lemma~\ref{lem:eichler-inv-minus1}. Equation~(\ref{eq:164}) is a
reformulation of Lemma~\ref{lem:prim-criterion}, and the last
statement is already contained in  Corollary~\ref{cor:2case-Eop}. 
\end{proof}

When combined with the criterion for $K\subseteq \Sigma_\scrG$ in
Lemma~\ref{lem:K-in-Sigma}, Theorem~\ref{thm:selectivity} reduces the
selectivity problem to purely local studies of quaternion orders over
complete discrete valuation rings. It also singles out the finitely
many places $\grp$ of $F$ with $e_\grp(\calO)=0$ as the possible
places of obstruction for the condition $K\subseteq \Sigma_\scrG$ to
be sufficient for selectivity. This point of view has already borne
fruit in \cite{peng-xue:select}, where Deke Peng and the first named
author applied Theorem~\ref{thm:selectivity} to obtain a selectivity
theorem for quaternion Bass orders that are well-behaved at the dyadic
primes of $F$. This provides the first known (as far as we are aware) concrete examples for
which the condition $K\subseteq \Sigma_\scrG$ is insufficient for
selectivity. Nevertheless, if we impose the condition that
$e_\grp(\calO)\neq 0$ for every finite prime $\grp$, then we get
results that look exactly like the ones in
\cite[Theorem~1.1]{M.Arenas-et.al-opt-embed-trees-JNT2018} and
\cite[Theorem~31.1.7]{voight-quat-book}, which are stated for Eichler
orders.

\begin{cor}\label{cor:non-zero-eichler-inv}
  Keep the assumption of Theorem~\ref{thm:selectivity} and assume further that
$e_\grp(\scrG)\neq 0$ for every finite prime $\grp$ of $F$. Then $B$ is
optimally spinor selective for $\scrG$ if and only if $K\subseteq
\Sigma_\scrG$. 
\end{cor}
   Similarly, the following lemma is a generalization of 
\cite[Proposition~31.2.1]{voight-quat-book} (see also
\cite[Theorem~1.1(3)]{M.Arenas-et.al-opt-embed-trees-JNT2018} and 
  \cite[Proposition~5.11]{Linowitz-Selectivity-JNT2012}).

\begin{lem}\label{lem:eich-inv-nonzero}
Suppose that $e_\grp(\scrG)\neq 0$ for
every finite prime $\grp$ of $F$.  Let $\grd(\calO)$ be the reduced
discriminant of an order  $\calO\in \scrG$. 
Then 
  $K\subseteq \Sigma_\scrG$  if and only if both of the following conditions hold:
\begin{enumerate}
\item[(a)] the extension $K/F$ and the $F$-algebra $D$ are unramified
  at every finite prime $\grp$ of $F$  and ramify at exactly the same  (possibly empty) set
  of infinite places;
\item[(b)] if $\grp$ is a finite prime of $F$ with $\nu_\grp(\grd(\calO))\equiv
  1\pmod{2}$, then $\grp$ splits in $K$. 
\end{enumerate}
\end{lem}
\begin{proof}
  First note that     $\Nr(\calO_\grp^\times)=O_{F_\grp}^\times$ for
  every finite prime $\grp$ of $F$.  
Indeed, if
  $e_\grp(\calO)=-1$, then this  has been proved in (\ref{eq:263});  if $e_\grp(\calO)\in \{1, 2\}$, then $\calO_\grp$ is an
Eichler order, so this is obvious. 
  Next it follows from
  \cite[Theorem~2.2]{Brzezinski-crelle-1990} that
  \begin{equation}
    \label{eq:266}
    \Nr(\calN(\calO_\grp))=\begin{cases}
      F_\grp^{\times2}O_{F_\grp}^\times &\text{if }
      \nu_\grp(\grd(\calO))\equiv 0\pmod{2}, \\
      F_\grp^\times &\text{if } \nu_\grp(\grd(\calO))\equiv
      1\pmod{2}. 
      \end{cases}
    \end{equation}
From Lemma~\ref{lem:K-in-Sigma}, we find that $K\subseteq
\Sigma_\scrG$ if and only if
\begin{itemize}
\item $K$ and $\grD$ are ramified at
  exactly the same (possibly empty) set of real places of $F$;
\item $K$ is unramified at every prime $\grp\subset O_F$ with
  $\nu_\grp(\grd(\calO))\equiv 0\pmod{2}$;
\item   $K$ splits at every prime $\grp\subset O_F$ with $\nu_\grp(\grd(\calO))\equiv
      1\pmod{2}$. 
\end{itemize}
Thus conditions (a) and (b) are clearly sufficient.  For the
necessity, recall that if $e_\grp(\calO)\in \{1, 2\}$, then
 $D_\grp\simeq M_2(F_\grp)$.  Thus non-zeroness assumption on the
 Eichler invariants implies that
$e_\grp(\calO)=-1$ at every finite prime $\grp$ ramified in $D$.  On
the other hand, if $D_\grp$ is
  division and $e_\grp(\calO)=-1$, then   $\nu_\grp(\grd(\calO))$ is
  odd by \cite[Corollary~3.2]{Brzezinski-1983}.    Thus if $K\subseteq
\Sigma_\scrG$, then $K$ splits at every $\grp$ ramified in $D$. This
forces $D$ to be unramified at all finite places of $F$ since $K$ embeds into $D$ by our
assumption. The conditions are necessary as well.
\end{proof}

% Nevertheless,
 % the theory remains simple if $\calO$ has nonzero Eichler
 % invariants everywhere. 

%   \begin{enumerate}
% \item for any two $O_F$-orders
%   $\calO, \calO'\in \scrG$,  
%    \begin{equation}
% \label{eq:164}
%     \Delta(B, \calO')=\big(\rho(\calO', \calO),
% K/F\big)+\Delta(B, \calO), 
%   \end{equation}
%   where $\big(\rho(\calO', \calO), K/F\big)\in \Gal(K/F)$ is the Artin
%   symbol as discussed right below (\ref{eq:166}), and the
%   summation is taken inside $\zmod{2}$ with the canonical
%   identification $\Gal(K/F)\simeq \zmod{2}$;

%   \item exactly half of the spinor genera in the genus $\scrG$ are selected
%     by $B$, and similarly, exact half of the spinor classes in $\SCl(\calO)$ is selected by $B$. 
%   \end{enumerate}

% If further $\grn$ is
%   square-free, then $\rho(\calO, \calO')\in \grG_\sg(\grn)$ is
%   represented by the distance ideal of $\calO$ and $\calO'$ (i.e.~the
%   index ideal $\chi(\calO/\calO\cap \calO')$ with the notation
%   following that of \cite[\S III.1]{Serre_local}); see
%   \cite[p.~34]{Chinburg-Friedman-1999} or \cite[(5),
%   p.~2855]{Maclachlan-selectivity-JNT2008}.

If $e_\grp(\scrG)\neq 0$ for
every finite prime $\grp$ of $F$ and $K\subseteq \Sigma_\scrG$, then for any two orders $\calO,
\calO'\in \scrG$, the element $\rho(\calO, \calO')|_K\in \Gal(K/F)$ can be
computed using an ideal  $I$ linking $\calO$ and $\calO'$ as in
Example~\ref{ex:link-ideal} since we do have
$\Nr(\calO_\grp^\times)=O_{F_\grp}^\times$ for every finite prime
$\grp$.

\section{Extending Maclachlan's relative conductor formula}
\label{sec:maclachlan}
Let $(B,\scrG)$ be as in the previous section, particularly the order $B\subset K$ satisfies condition \eqref{eq:3}. Throughout this section, we let $\scrG=\scrG_\grn$ be the genus of
Eichler orders in $D$ of level $\grn\subseteq O_F$.  For simplicity,
put
$\Sigma_\grn:=\Sigma_{\scrG_\grn}$, and assume that
$K\subseteq \Sigma_\grn$ so that every order in $K$ satisfying condition 
(\ref{eq:3}) is selective for $\scrG_\grn$. 
Let $\grf(B)$ be
the conductor of $B$, i.e.~$\grf(B)$ is the unique ideal of $O_F$ such
that $B=O_F+\grf(B)O_K$. If $B'\subseteq O_K$ is another $O_F$-order
in $K$, we put $\grf(B'/B):=\grf(B')^{-1}\grf(B)$ and call it the
\emph{relative conductor} of $B$ with respect to $B'$. This is a
fractional ideal of $O_F$. Assume that $B'$ satisfies condition
(\ref{eq:3}) as well.    If $\grn$ is square-free, Maclachlan shows in
\cite[Theorem~3.3]{Maclachlan-selectivity-JNT2008} that for any two
orders $\calO, \calO'\in \scrG_\grn$, the selectivity symbol formula
(\ref{eq:164}) may be enhanced to
   \begin{equation}
    \label{eq:120}
    \Delta(B, \calO)=\big(\grf(B'/B),
    K/F\big)+\rho(\calO, \calO')|_K+\Delta(B', \calO').  
  \end{equation}
Here  $\big(\grf(B'/B),
    K/F\big)\in \Gal(K/F)$ is the Artin symbol \cite[\S X.1]{Lang-ANT}, which is well-defined
    since the assumption $K\subseteq \Sigma_\grn$ forces $K/F$ to be unramified at all the finite places of $F$
    by Lemma~\ref{lem:eich-inv-nonzero}.  The
  summation on the right hand side is taken inside $\zmod{2}$ with the canonical
  identification $\Gal(K/F)=\zmod{2}$ as before.   In this section, we
  show that Maclachlan's formula holds without the square-free
  assumption on $\grn$.

  First, let us workout condition (\ref{eq:3}) more concretely in
  terms of $\grn$ and $\grf(B)$. Since $K\subseteq \Sigma_\grn$ by our
  assumption, $D$ and $K$ are unramified at all the finite places of $F$ according to
  Lemma~\ref{lem:eich-inv-nonzero}. Particularly, the reduced discriminant $\grd(\calO)$ of $\calO$ coincides with the level $\grn$.  For each finite prime $\grp$ of $F$, let
  $\nu_\grp: F_\grp^\times\twoheadrightarrow \bbZ$ be the normalized
  discrete valuation. Define  numerical invariants of $\calO$ and
  $B$ at $\grp$ as follows: 
  \begin{equation}
    \label{eq:15}
n_\grp:=\nu_\grp(\grd(\calO))=\nu_\grp(\grn), \qquad i_\grp(B):=\nu_\grp(\grf(B)). 
  \end{equation}
 We
  quote the relevant part of the criterion for the existence of local
  optimal embeddings given by Guo and Qin in
  \cite[Lemma~2.2]{Guo-Qin-embedding-Eichler-JNT2004}, which in turn is 
  based on a theorem of Brzezinski
  \cite[Theorem~1.8]{Brzezinski-embed-num-1991}. 
\begin{lem}\label{lem:Guo-Qin}
Keep the assumption that $\grp$ is unramified in $K$.  Then $\Emb(B_\grp, \calO_\grp)\neq \emptyset$ if and only if one of the
following conditions holds
\begin{enumerate}
\item $(K/\grp)=1$;
\item   $(K/\grp)=-1$ and $n_\grp\leq 2i_\grp(B)$. 
\end{enumerate}
In particular, if $\Emb(B_\grp, \calO_\grp)\neq \emptyset$, then
$\Emb(B_\grp', \calO_\grp)\neq\emptyset$ for any order
$B_\grp'\subseteq B_\grp$. 
\end{lem}

% Let $B\subseteq O_K$
% be an $O_F$-order in $K$ of conductor $\grf(B)$. In other words,
% $\grf(B)$ is the unique ideal of $O_F$ such that
% $B=O_F+\grf(B)O_K$. If $B'\subseteq O_K$ is another $O_F$-order in
% $K$, we put 
% \begin{equation}
%   \label{eq:122}
% \grf(B'/B):=\grf(B')^{-1}\grf(B)\in \calI_F, 
% \end{equation}
% and call it the \emph{relative conductor} of $B$ with respect to
% $B'$. Indeed, if $B\subseteq B'$, then $\grf(B'/B)$ is the global
% version of the relative
% conductor defined in \cite[\S II.3, p.~44]{vigneras}. In general,
% $\grf(B'/B)$ is a fractional ideal. 

% \begin{rem}\label{rem:rel-conductor}
%  Here $\big(\grf(B'/B),
%     K/F\big)\in \Gal(K/F)$ is the Artin symbol, . In
%     Proposition~\ref{}, we will show that (\ref{eq:120}) holds in
%     general for Eichler orders of arbitrary levels as
%     well. This is particularly useful for calculation purpose since in
%     general it is difficult to produce another order $\calO'\in \scrG$
%     with a  known value $\Delta(B, \calO')$. However, one can always
%     choose a suitable embedding $\varphi: K\to D$ and compute $B':=
%     \varphi^{-1}(\calO)$. By construction,  $\Delta(B',\calO)=1$, so 
%     (\ref{eq:120}) reduces to $\Delta(B, \calO)=\big(\grf(B'/B),
%     K/F\big)+1$.  Unfortunately, (\ref{eq:120}) does not extend to
%     orders with nonzero Eichler invariants everywhere, as shown in
%     Remark~\ref{}. 
% \end{rem}

%Our main theorem for this section is as follows. 

\begin{thm}\label{thm:maclachlan}
  Let $\grn$  be an arbitrary ideal of $O_F$. Then formula (\ref{eq:120}) holds for any  orders $\calO, \calO'\in \scrG_\grn$
  and for any orders $B, B'$ in $K$ both satisfying condition (\ref{eq:3}).  
\end{thm}

% We obtain the following corollary by choosing a suitable $B'$ and taking
% $\calO'=\calO$.

In particular, we can take $\calO'=\calO$ and pick a suitable $B'$ to
compute $\Delta(B, \calO)$.
\begin{cor}
Keep $\grn, B, \calO$ be as in Theorem~\ref{thm:maclachlan}.  Let $B':=\varphi^{-1}(\calO)$ for an
   $F$-embedding $\varphi: K\to D$.  Then
  \[\Delta(B, \calO)=\big(\grf(B'/B),
    K/F\big)+1.\]
\end{cor}
\begin{proof}[Proof of Theorem~\ref{thm:maclachlan}]
As usual, all the orders in $K$ (such as $B, B', B_0$
etc.) considered below 
 are assumed to satisfy condition (\ref{eq:3}).   
 We make a couple of simplifications.  
\begin{enumerate}[label=(\alph*)]
\item If   (\ref{eq:120}) holds for a fixed pair of orders $\calO_0,
  \calO_0'\in \scrG$, then it holds for every pair $\calO, \calO'\in
  \scrG$. This follows directly from (\ref{eq:164}) and the additive
  property of $\rho(\calO, \calO')$  in (\ref{eq:26}). 
\item If (\ref{eq:120}) holds
  for $B=B_0$ for a fixed order $B_0$ (with $B'$ still being arbitrary), then it holds for every $B$. This again follows  from
  (\ref{eq:164}) and the multiplicative property of the relative
  conductors: 
  \[\grf(B'/B)=\grf(B'/B_0)\grf(B_0/B).\]
\end{enumerate}

Now  pick $B_0$ to be the unique order in $K$ satisfying 
\begin{equation}
  \label{eq:49}
  i_\grp(B_0)=
  \begin{cases}
    0 &\text{ if $n_\grp$ is odd,}    \\
    n_\grp& \text{ if $n_\grp$ is even}. 
  \end{cases}
\end{equation}
Since $n_\grp=0$ for almost all $\grp$, such an order does
exist in $K$. Moreover, it satisfies condition (\ref{eq:3}) by Lemma~\ref{lem:eich-inv-nonzero} and 
Lemma~\ref{lem:Guo-Qin}. Thus there exists an order $\calO_1\in
\scrG_\grn$ that admits an  optimal embedding 
$\varphi\in \Emb(B_0, \calO_1)$.  We
shall modify $\calO_1$ locally at finite many places to obtain another
order $\calO_0\in\scrG_\grn$ that is more conductive to our purpose.
For each $\varepsilon\in \{\pm 1\}$, let  $T_\varepsilon$  be
the following finite set of primes $\grp\subset O_F$:
\begin{equation}
  \label{eq:52}
T_\varepsilon:=\{\grp\mid (K/\grp)=\varepsilon \text{ and } i_\grp(B')\neq i_\grp(B_0)  \}.
\end{equation}
It follows from Lemma~\ref{lem:eich-inv-nonzero} and the assumption
$K\subseteq \Sigma_\grn$ that $n_\grp$ is even for every
$\grp\in T_{-1}$.  In particular, $i_\grp(B_0)=n_\grp$ for every $\grp\in
T_{-1}$ by definition.

At each prime $\grp\subset O_F$ inert in $K$, there exists an element
$\omega_\grp\in O_{K_\grp}^\times$ satisfying a quadratic equation
\begin{equation}
  \label{eq:53}
  \omega_\grp^2-t_\grp \omega_\grp+u_\grp=0, \qquad \text{with}\qquad 
  t_\grp\in O_{F_\grp}, \, u_\grp\in O_{F_\grp}^\times 
\end{equation}
such that $O_{K_\grp}=O_{F_\grp}+\omega_\grp O_{F_\grp}$.  For each
$\grp\in T_{-1}$, we choose a suitable identification $D_\grp=M_2(F_\grp)$ such that
$\varphi: K_\grp\to D_\grp$ sends $\omega_\grp$ to the element $
\begin{bmatrix}
  0 & -u_\grp \\ 1 & t_\grp
\end{bmatrix}$, where $\varphi\in \Emb(B_0, \calO_1)$ is the optimal
embedding as above.  Let $\pi_\grp$ be a uniformizer in $F_\grp$. 
There is a unique order $\calO_0\in \scrG_\grn$
satisfying the following local properties:
\[ (\calO_0)_\grp=
  \begin{bmatrix}
    O_{F_\grp}&     O_{F_\grp}\\
      \pi_\grp^{n_\grp}  O_{F_\grp}&     O_{F_\grp}
  \end{bmatrix}, \quad \forall \grp\in T_{-1}, \qquad \text{and}\qquad
  (\calO_0)_\grp=(\calO_1)_\grp, \quad \forall \grp\not\in T_{-1}. 
\]
By our construction, for every $\grp\in T_{-1}$ we have 
\[\varphi(K_\grp)\cap
  (\calO_0)_\grp=\varphi(O_{F_\grp}+\pi_\grp^{n_\grp}\omega_\grp
  O_{F_\grp})=\varphi((B_0)_\grp), \]
so $\varphi$ defines  an optimal embedding of $B_0$ into $\calO_0$. In
particular, $\Delta(B_0,\calO_0)=1$.

Next, we produce another order $\calO_0'$ such that $\varphi$ defines
an optimal embedding of $B'$ into $\calO_0'$. For each $\grp$ in $T_1$,
we pick an element $x_\grp\in D_\grp^\times$ such that $\varphi\in
\Emb(B_\grp', x_\grp(\calO_0)_\grp x_\grp^{-1})$. Such an $x_\grp$
exists by part (1) of Lemma~\ref{lem:Guo-Qin}. For each $\grp\in T_{-1}$,
we put $x_\grp=
\begin{bmatrix}
  0 & 1 \\ \pi_\grp^{i_\grp'} & 0
\end{bmatrix}$, where $i_\grp':=i_\grp(B')$. One easily computes that
\[x_\grp (\calO_0)_\grp x_\grp^{-1}=
  \begin{bmatrix}
    O_{F_\grp}&     \pi_\grp^{n_\grp-i_\grp'} O_{F_\grp}\\
      \pi_\grp^{i_\grp'}  O_{F_\grp}&     O_{F_\grp}
  \end{bmatrix}.
\]
We have $n_\grp-i_\grp' \leq i_\grp'$ by part (2) of
Lemma~\ref{lem:Guo-Qin}.  Since $u_\grp\in O_{F_\grp}^\times$ by
construction, it follows that 
\[\varphi(K_\grp)\cap
  x_\grp (\calO_0)_\grp x_\grp^{-1}=\varphi(O_{F_\grp}+\pi_\grp^{i_\grp'}\omega_\grp
  O_{F_\grp})=\varphi(B_\grp'),  \qquad \forall \grp\in T_{-1}.    \]
Lastly, we define $\calO_0'$ to be the unique order such that
$\wcO_0'=x \wcO_0 x^{-1}$, where $x\in \whD^\times $ is the element such that
$x_\grp$ is defined above if $\grp\in T_1\cup T_{-1}$, and $x_\grp=1$
at every other $\grp$. The construction above guarantees that
$\varphi\in \Emb(B', \calO_0')$, so $\Delta(B', \calO_0')=1$. 

Now to show that formula (\ref{eq:120}) holds for $B=B_0,
\calO=\calO_0$ and $\calO'=\calO'_0$, it is enough to check that 
\begin{equation}
  \label{eq:54}
\big(\grf(B'/B_0),
    K/F\big)+\rho(\calO_0, \calO_0')|_K=0.  
\end{equation}
From (\ref{eq:49}), we have $\grf(B'/B_0)=\prod_{\grp\in T_{-1}}
\grp^{n_\grp-i_\grp'}\prod_{\grp\in T_1}
\grp^{i_\grp(B_0)-i_\grp'}$.  Since $(\grp, K/F)$ vanishes if $\grp\in
T_1$ and  $n_\grp$ is even at every $\grp\in T_{-1}$, we obtain
\begin{equation}
  \label{eq:55}
\big(\grf(B'/B_0),
    K/F\big)=\sum_{\grp\in T_{-1}} (\grp^{-i_\grp'}, K/F),    
\end{equation}
where $\Gal(K/F)$ is identified with $\zmod{2}$ as usual.
 On the other hand, 
 \begin{equation}
   \label{eq:56}
\rho(\calO_0, \calO_0')|_K=(\Nr(x), K/F)=\sum_{\grp\in T_{-1}}
  (\grp^{i_\grp'}, K/F),   
 \end{equation}
where the last equality holds because those $x_\grp$ for $\grp\in T_1$
make
no contribution to the Artin
symbol.  Now (\ref{eq:54}) follows from combining (\ref{eq:55})
and (\ref{eq:56}), and the Theorem is proved by the observations at
the start of the proof. 
\end{proof}

\section{The spinor trace formula}\label{sec:spinor-trace-formula}

In this section, we assume that  $\scrG$ is an arbitrary genus  of
orders in $D$, and $K$ is not necessarily contained in
$\Sigma_\scrG$.  Let $\calO$ be an order in $\scrG$,  and  $B$ be an order in $K$
satisfying condition (\ref{eq:3}). Recall from (\ref{eq:22}) that $m(B, \calO,
\calO^\times)$ denotes the number of optimal embeddings of $B$ into
$\calO$ up to $\calO^\times$-conjugacy. Our goal of this section is to
generalize the Vign\'eras-Voight formula \cite[Corollary~31.1.10]{voight-quat-book} for $m(B, \calO,
\calO^\times)$ to the totally definite case. Let us first set up some
notations and reproduce this
formula for the convenience of the reader. 

 By definition, the class number of
$\calO$ (denoted as $h(\calO)$) is the cardinality of the finite set
$\Cl(\calO)$ of locally principal right $\calO$-ideal classes
in $D$.  It is well known that $h(\calO)$ depends only on the genus
$\scrG$ and not on the choice of $\calO\in \scrG$. Similarly, let
$h(B)$ denote the class number of $B$. We define the symbol
\begin{equation}
  s(K, \scrG)=
  \begin{cases}
    1 & \text{if }  K\subseteq \Sigma_\scrG;\\
    0 & \text{otherwise}. 
  \end{cases}
\end{equation}

\begin{prop}[Vign\'eras-Voight]\label{prop:vv}
  Suppose that $D$ satisfies the Eichler condition, and $\scrG$ is a
  genus of Eichler orders. Then
  \begin{equation}
    \label{eq:11}
    m(B, \calO, \calO^\times)=\frac{2^{s(K, \scrG)}\Delta(B,
      \calO)h(B)}{h(\calO)}\prod_{\grp} m_\grp(B) ,  
\end{equation}
where $m_\grp(B):=m(B_\grp, \calO_\grp,
  \calO_\grp^\times)$, and the product runs over all primes
  $\grp\subset O_F$. 
\end{prop}

The product on the right hand side of \eqref{eq:11} is well-defined since $m_\grp(B)=1$ for all but
  finitely many $\grp$ according to
  \cite[Theorem~II.3.2]{vigneras}.  The class number $h(B)$ can be
  computed using Dedekind's formula \cite[p.~75]{vigneras:ens} (See
  also \cite[\S7]{li-xue-yu:unit-gp} and \cite[\S5, p.~804]{xue-yu:type_no}). 
Moreover, under the assumptions of
  Proposition~\ref{prop:vv}, the class number $h(\calO)$ coincides
  with the \emph{restricted class number $h_D(F):=\abs{F_D^\times\bsh \whF^\times/\whO_F^\times}$ of $F$ with respect to
  $\grD$} as in  \cite[Corollary~III.5.7]{vigneras}. The formula for
the local optimal embedding numbers 
$m_\grp(B)$ has been work out by many people in various generality,
 notably by Eichler \cite{eichler:crelle55}, Hijikata
 \cite{hijikata:JMSJ1974}, Pizer\cite{pizer:2},  and Brzezinski
 \cite{Brzezinski-embed-num-1991, Brzezinski-crelle-1990}.  See
 \cite[Lemmas~30.6.16--17]{voight-quat-book} for the exposition in the
 case of Eichler orders.

A priori, (\ref{eq:11}) is only stated in
\cite[Corollary~31.1.10]{voight-quat-book} for the non-selective case
(i.e.~$s(K, \Sigma_\scrG)=0$ so that $\Delta(B, \calO)=1$ for every
$\calO\in \scrG$). Nevertheless, the same proof applies to the
selective case as well.   Indeed, the proof combines
Theorem~\ref{thm:Voight31} with the following 
classical  \emph{trace formula}
\cite[Theorem~III.5.11]{vigneras} (cf.~\cite[Lemma~3.2.1]{xue-yang-yu:ECNF} and
\cite[Lemma~3.2]{wei-yu:classno}) for optimal embeddings:
  \begin{equation}
    \label{eq:51}
\sum_{[I]\in \Cl(\calO)} m(B,\calO_l(I),
    \calO_l(I)^\times)=h(B)\prod_\grp m_\grp(B). 
  \end{equation}
Here $[I]$ denotes the right $\calO$-ideal class of $I$, and  $\calO_l(I)$ is the left order of $I$ defined in
(\ref{eq:12}). 
The trace formula holds for any quaternion algebras,
i.e.~it does not require the Eichler condition.

% The association of each $I$ with its left order $\calO_l(I)$ induces a
% well-defined and surjective map:
% \begin{equation}
%   \label{eq:57}
% \Cl(\calO)\twoheadrightarrow \Tp(\calO), \qquad [I]\to [\calO_l(I)].    
% \end{equation}

  %and $\Cl(\calO)$
% be the finite set of of locally principal right $\calO$-ideal classes
% in $D$. The cardinality of $\Cl(\calO)$ is called the \emph{class number} of
% $\calO$  and denoted by $h(\calO)$. 

As remarked by Voight \cite[Remark~31.6.2]{voight-quat-book}, in the
totally definite case it is no longer viable to produce a formula for
 $m(B, \calO, \calO^\times)$ for every single $\calO\in \scrG$.
 Instead, we group the right $\calO$-ideal classes into \emph{spinor
   classes} and produce a refinement of the trace formula.

 \begin{defn}
Two locally
  principal right $\calO$-ideals $I$ and $I'$ are in the   \emph{same
spinor class} if there exists  $x\in
\grD^\times\whD^1$ such that $\wh I'= x \wh I$. 
  \end{defn}

Note that if $I$ and $I'$ belong to the same spinor class, then their
left orders $\calO_l(I)$ and $\calO_l(I')$ are in the same spinor
genus.   
The spinor class of $I$ is
  denoted by $[I]_\scc$. For a fixed spinor class $[I]_\scc$, the set
  of ideal classes in  $[I]_\scc$ is denoted by $\Cl(\calO,
  [I]_\scc)$. In other words, 
  \begin{equation}
    \label{eq:125}
    \Cl(\calO,
  [I]_\scc):=\{[I']\in \Cl(\calO)\mid [I']\subseteq [I]_\scc\}.
  \end{equation}
For simplicity, we put $\Cl_\scc(\calO):=\Cl(\calO, [\calO]_\scc)$.

  Let $\SCl(\calO)$ be the finite set of spinor classes of locally principal
     right $\calO$-ideals. It can be described adelically as follows
\begin{equation}
  \label{eq:14}
  \SCl(\calO)\simeq (\grD^\times\whD^1)\bsh \whD^\times/\wcO^\times\xrightarrow[\simeq]{\Nr}
  F_\grD^\times\bsh \whF^\times/\Nr(\wcO^\times),   
\end{equation}
where the  two double coset spaces are canonically bijective via
the reduced norm map.
This equips $\SCl(\calO)$ with an abelian group structure (whose
identity element is $[\calO]_\scc$),   so we
call it  the \emph{spinor class group} of $\calO$. In light of the
adelic description of $\SG(\calO)$ in \eqref{eq:118}, there is a canonical surjective
group homomorphism
\begin{equation}
  \label{eq:58}
   \SCl(\calO)\twoheadrightarrow \SG(\calO),\qquad
  [I]_\scc\mapsto [\calO_l(I)]_\sg. 
\end{equation}

Our spinor trace formula can be stated as follows. 

\begin{prop}[Spinor trace formula]\label{prop:spinor-trace-formula}
Let $\scrG$ be an arbitrary genus of orders in $D$. 
  Suppose that either $K\cap \Sigma_\scrG=F$ or $B$ is selective for $\scrG$.
Then for each  $[J]_\scc\in
\SCl(\calO)$, we have 
\begin{equation}
\label{eq:251}
\sum_{[I]\in \Cl(\calO, [J]_\scc)}
m(B,\calO_l(I),
    \calO_l(I)^\times)=\frac{2^{s(K, \scrG)}\Delta(B,
      \calO_l(J))h(B)}{\abs{\SCl(\calO)}}\prod_{\grp} m_\grp(B). 
\end{equation}
\end{prop}

\begin{rem}\label{rem:nonzero}
We make a few observations.
  \begin{enumerate}[label=(\alph*)]
  \item If $\grD$ is
  further assumed to satisfy the Eichler condition, then
  $\Cl(\calO, [J]_\scc)$ is a singleton with the unique member $[J]$
  by \cite[Proposition~1.1]{Brzezinski-Spinor-Class-gp-1983},
  so we recover a slightly more general form of the Vign\'eras-Voight
  formula  (\ref{eq:11}).

  \item   From Corollary~\ref{cor:non-zero-eichler-inv}, if
  $e_\grp(\calO)\neq 0$ for every finite prime $\grp$ of $F$ (e.g.~if  $\calO$ is an Eichler order), then the
  assumption of the proposition holds automatically.
  Moreover, in this case $\abs{\SCl(\calO)}$ is equal to the
  restricted class number $h_D(F)$ of $F$ with respect to
  $D$ since $\Nr(\calO_\grp^\times)=O_{F_\grp}^\times$ for every
   prime $\grp\subset O_F$. 
\item Form Corollary~\ref{cor:2case-Eop}, $B$ is selective for the genus $\scrG$
  if and only if $E_\op=K$, where $E_\op$ is the optimal
  representation field in
  Definition~\ref{defn:op-rep-field}. Recall from  (\ref{eq:5})  that
  we have the following chain of inclusions: 
  \[     F\subseteq E_\op\subseteq K\cap \Sigma_\scrG\subseteq K. \]
The assumption of the proposition covers the two cases below: \[F=E_\op=K\cap
  \Sigma_\scrG \subsetneq K, \qquad F\subsetneq E_\op=K\cap
  \Sigma_\scrG= K.\] In the last remaining case where
$F=E_\op\subsetneq K\cap \Sigma_\scrG= K$, it is not known whether
(\ref{eq:251}) still  holds true or not. Examples of $(B,\calO)$ falling
into this third case  do exist, as shown by Peng and the first named
author in \cite[\S5]{peng-xue:select}. 
 \end{enumerate}
\end{rem}

To prove Proposition~\ref{prop:spinor-trace-formula}, 
we reformulate the summation in (\ref{eq:251}) 
 adelically. If $\whI=x\wcO$ for some $x\in \whD^\times$, then we
put 
\begin{equation}
  \label{eq:258}
[\whI]_\scc:=\grD^\times\whD^1x\wcO^\times=\grD^\times\whD^1\wcO^\times
x,
\end{equation}
where the last equality follows from the fact that
\begin{equation}
  \label{eq:255}
  \whD^1\wcO^\times=  \whD^1x\wcO^\times x^{-1}, \qquad \forall x\in \whD^\times. 
\end{equation}
Clearly,  $[\whI]_\scc$ 
depends only on the spinor class
$[I]_\scc$.
The set  $\Cl(\calO,
  [I]_\scc)$ may be described adelically as 
\begin{equation}
  \label{eq:126}
  \begin{split}
   \Cl(\calO,  [I]_\scc)&\simeq \grD^\times\bsh [\whI]_\scc/\wcO^\times\simeq \grD^\times\bsh (\grD^\times \whD^1
   x \wcO^\times x^{-1})/(x\wcO^\times x^{-1}), 
  \end{split}
\end{equation}
where the last isomorphism is induced from the right multiplication of
$[\whI]_\scc$ by $x^{-1}$. 
If we set $\calO':=\calO_l(I)$ so that $\wcO'=x\wcO x^{-1}$, then there is a
bijection 
\begin{equation}
  \label{eq:229}
   \Cl(\calO,  [I]_\scc)\simeq  \Cl(\calO',  [\calO']_\scc)=:\Cl_\scc(\calO').
 \end{equation}
  Fix an embedding $\varphi: K\hookrightarrow \grD$ and
 identify $K$ with its image in $\grD$ as before.  At the moment, we
 do not require $\varphi\in \Emb(B, \calO)$ yet. Let $\whE:=\whE(\varphi, B,
\calO)$ be as defined in (\ref{eq:158}). Using the same method 
 as in the proofs of
 \cite[Theorem~III.5.11]{vigneras} and
 \cite[Theorem~30.4.7]{voight-quat-book}, we immediately obtain the
 following lemma. 

\begin{lem}\label{lem:adelic-trace}
We have 
\begin{equation}\label{eq:137}
  \sum_{[I]\in \Cl(\calO, [J]_\scc)} m(B,\calO_l(I),
    \calO_l(I)^\times)=  \abs{K^\times\bsh(\whE\cap [\whJ]_\scc)/\wcO^\times}. 
\end{equation}
\end{lem}
%This lemma can be proved by the same method so we omit its proof here. 
% \begin{proof}
% The lemma can be obtained by the same method as in the proof of
% \cite[Theorem~III.5.11]{vigneras} or \cite[Theorem~30.4.7]{voight-quat-book}. We provide a brief sketch for the
% convenience of the reader. 

% Let $\{g_1, \cdots, g_r\}\subset \whD^\times$ be a complete set
% of representatives of the double coset
% space $\grD^\times\bsh [\whJ]_\scc/\wcO^\times$. 
% Then $\{I_i:=g_i\calO\mid 1\leq i\leq r\}$ forms a complete set of
% representatives of $\Cl(\calO, [J]_\scc)$ by (\ref{eq:126}). Put $\calO_i:=\calO_l(I_i)$
% for each $1\leq i\leq r$. 
% There is a canonical bijection 
%  \begin{equation}
%    \label{eq:46}
%  K^\times\bsh\calE(\varphi, B, \calO_i)\to \Emb(B, \calO_i), \qquad K^\times \alpha\mapsto
%  \alpha^{-1}\varphi \alpha,  
% \end{equation}
% where $\calE(\varphi, B, \calO_i)$ is defined as in (\ref{eq:25}). 
% On the other hand, consider the map 
% \begin{align}
%   \label{eq:47}
%   \Phi: \bigsqcup_{i=1}^r  K^\times\bsh\calE(\varphi, B,
%   \calO_i)/\calO_i^\times&\to K^\times\bsh (\whE\cap [\whJ]_\scc)/\wcO^\times,\\
% K^\times \alpha \calO_i^\times&\mapsto K^\times \alpha g_i\wcO^\times. 
% \end{align}
% It is straightforward to check by definition that  $\Phi$ is a well-defined bijection,
% and the lemma follows.
% \end{proof}

We leave it as an exercise using Lemma~\ref{lem:adelic-trace}  to show that if $\calO$ and  $\calO' $
belong to the same spinor genus, then 
\begin{equation}
  \label{eq:211}
  \sum_{[I]\in \Cl_\scc(\calO)} m(B,\calO_l(I),
    \calO_l(I)^\times)=  \sum_{[I']\in \Cl_\scc(\calO')} m(B,\calO_l(I'),
    \calO_l(I')^\times). 
\end{equation}

 % By (\ref{eq:126}), 
% \begin{equation}
%   \label{eq:156}
%   \Cl(\calO, [J]_\scc)\simeq
%   \grD^\times\bsh\whD([J]_\scc)/\wcO^\times. 
% \end{equation}

% Now we restrict to Eichler orders of square-free
% level. Let $\grn\subseteq O_F$ be a square-free ideal coprime to
% $\grd(\grD)$, and $\scrG_\grn$ be the genus of Eichler orders of level
% $\grn$.  We provide a precise formula for the left hand side of
% (\ref{eq:137}) in the case of $\calO\in \scrG_\grn$. Let us put
%   \begin{equation}
%     \label{eq:153}
%     \wtM(B, \grn):=h(B)\prod_{\grp\ddiv\grd(D)}\left(1-\Lsymb{B}{\grp}\right)\prod_{\grp\mid
%       \grn}\left(1+\Lsymb{B}{\grp}\right).
%   \end{equation}
% Thanks to (\ref{eq:28}), the formula (\ref{eq:51}) now takes the form 
% \begin{equation}
%   \label{eq:160}
% \sum_{[I]\in \Cl(\calO)} m(B,\calO_l(I),
%     \calO_l(I)^\times)=\wtM(B, \grn). 
% \end{equation}

%  one of the following conditions holds:
% \begin{enumerate}[label=(\roman*)]
% \item $m_\grp(B)=0$ for some $\grp\in \Pl_f(F)$;
% \item $K\cap \Sigma_\scrG=F$;
% \item $B$ is selective for $\scrG$.
% \end{enumerate}

Now we are ready to prove the spinor trace formula. 

\begin{proof}[Proof of Proposition~\ref{prop:spinor-trace-formula}]
Recall that  the canonical map $\SCl(\calO)\twoheadrightarrow \SG(\calO)$ 
in (\ref{eq:58}) is a group homomorphism. From part (ii) of 
Corollary~\ref{cor:2case-Eop}, if $B$ is selective for $\scrG$, then
exactly half of the spinor classes $[I]_\scc\in \SCl(\calO)$ satisfy that
$\Delta(B, \calO_l(I))=1$. 
  In light of the trace formula (\ref{eq:51}), the 
  proposition reduces to a statement reminiscent of part (iii) of
  Theorem~\ref{thm:Voight31}:   
   if $\Delta(B, \calO_l(J))=\Delta(B,
\calO_l(J'))=1$ for two spinor classes $[J]_\scc$ and $[J']_\scc$ in $\SCl(\calO)$, then 
\begin{equation}
  \label{eq:132}
  \sum_{[I]\in \Cl(\calO, [J]_\scc)} m(B,\calO_l(I),
    \calO_l(I)^\times)=\sum_{[I']\in \Cl(\calO, [J']_\scc)} m(B,\calO_l(I'),
    \calO_l(I')^\times).
\end{equation}

Now let $\calO'\in
\scrG$ be another member of $\scrG$, and $M'$ be a locally principal right
$\calO'$-ideal with $\calO_l(M')=\calO$.  The map $I\mapsto
IM'$ induces a bijection between locally principal right ideals of
$\calO$ and those of $\calO'$, and it preserves ideal classes,
spinor classes, and associated left orders. Therefore,
\begin{equation}
  \label{eq:210}
  \sum_{[I]\in \Cl(\calO, [J]_\scc)} m(B,\calO_l(I),
    \calO_l(I)^\times)=  \sum_{[I']\in \Cl(\calO', [JM']_\scc)} m(B,\calO_l(I'),
    \calO_l(I')^\times).
\end{equation}
Replacing $\calO$ by a suitable $\calO'$ if necessary, we shall 
assume that there exists an optimal embedding $\varphi\in \Emb(B,
\calO)$. Clearly, if  (\ref{eq:132}) holds for
$J'=\calO$, then it holds for all $J'$ with   $\Delta(B,
\calO_l(J'))=1$, so we further take $J'=\calO$.

% It is enough to show that (\ref{eq:132})
% holds for $J'=\calO$. 

%  be an optimal embedding of $B$ into
% $\calO$, and put 
% \[B_0=\varphi_0(B), \qquad K_0=\varphi_0(K)\subset \grD.\]
% Consider the sets 
% \begin{align}
%   \label{eq:25}
%   \wh\calE(\varphi_0, B, \calO)&:=\{g=(g_\grp)\in \whD^\times \mid K_0\cap g\wcO
%   g^{-1}=B_0\}, \quad\text{and} \\
%   \calE_\grp(\varphi_0, B, \calO)&:=\{g_\grp\in \grD_\grp^\times \mid (K_0)_\grp\cap g_\grp\calO_\grp
%   g_\grp^{-1}=(B_0)_\grp\}, 
% \end{align}
% where $\grp$ is a finite prime of $F$. Since $\varphi_0$ is fixed
% throughout this proof, we drop it from the notation and write simply $\whE(B, \calO)$ and $\calE_\grp(B, \calO)$ instead. 

Keep the optimal embedding $\varphi\in \Emb(B, \calO)$ fixed and identify $K$
with $\varphi(K)\subseteq \grD$ as before.  By the assumption, either $K\cap \Sigma_\scrG=F$ or $B$ is selective for $\scrG$. This guarantees that 
\begin{equation}
  \label{eq:190}
F_\grD^\times
  \Nr(\whK^\times)\Nr(\calN(\wcO))=
  F_\grD^\times\Nr(\whE).
\end{equation}
Indeed, from (\ref{eq:173}) and (\ref{eq:5}),  if $K\cap \Sigma_\scrG=F$, then both sides are equal to
$\whF^\times$; if $B$ is selective for $\scrG$, then both sides are
equal to $F_K^\times\Nr(\whK^\times)$.

% $\whE(B, \calO)$ and $\calE_\grp(B, \calO)$ be as in (\ref{eq:158})
% and (\ref{eq:159}).  Similar to
% (\ref{eq:46}), there is a canonical bijection 
% \begin{equation}
%   \label{eq:45}
%  (K_0)_\grp^\times\bsh\calE_\grp(B,
%   \calO)\to \Emb(B_\grp,
%                                 \calO_\grp),\qquad
%                  (K_0)_\grp^\times g_\grp\mapsto  g_\grp^{-1}\varphi_0 g_\grp.  
% \end{equation}
% It follows from (\ref{eq:30}) that 
% \begin{equation}
%   \label{eq:133}
% \calE_\grp(B,
%   \calO)=(K_0)_\grp^\times\calN(\calO_\grp),\qquad   \wh\calE(B,
%   \calO)=\whK_0^\times\calN(\wcO). 
% \end{equation}
% Note that the second equality is equivalent to
% Theorem~\ref{thm:eichler} above.

%  On the other hand,
%   the double coset space $(\whK_0)^\times\bsh\wh\calE(B,
%   \bbO_0)/\calN(\wbO_0)$ can be decomposed locally: 
%   \begin{equation}
%     \label{eq:63}
% (\whK_0)^\times\bsh\wh\calE(B,
%   \bbO_0)/\calN(\wbO_0)\xrightarrow{\simeq} \prod_\grp (K_0)_\grp^\times\bsh\wh\calE_\grp(B,
%   \bbO_0)/\calN((\bbO_0)_\grp). 
%   \end{equation}
% Therefore, we have 
% \begin{equation}
%   \label{eq:31}
%   \wh\calE(B,
%   \bbO_0)=\whK_0^\times\calN(\wbO_0). 
% \end{equation}

% and put 
% \begin{align}
%   \label{eq:127}
% \whD([J]_\scc)&:=\grD^\times\whD^1z\wcO^\times,\\
%   \label{eq:41}
% \whE(B, \calO, [J]_\scc)&:=\whD([J]_\scc)\cap \wh\calE(B,
% \calO)=\whD([J]_\scc)\cap \whK_0^\times\calN(\wcO). 
% \end{align}
%\[  K_0\cap (yz)\wcO(y z)^{-1}=B_0. \]

Write $\widehat{J}=z\wcO$ for some $z\in \whD^\times$. Then
$\widehat{\calO_l(J)}=z\wcO z^{-1}$. By Lemma~\ref{lem:prim-criterion}, 
the
assumption $\Delta(B, \calO_l(J))=1$ implies that $\Nr(z)\in
F_\grD^\times\Nr(\whE)$. Thanks to (\ref{eq:190}), 
 there exist
 \begin{equation}
   \label{eq:191}
y\in \grD^\times\whD^1,\quad   
k\in \whK^\times\quad \text{and}\quad u\in
\calN(\wcO)
 \end{equation}
 such that $yz=k u$. Thus we have 
\begin{equation}
  \label{eq:134}
[\whJ]_\scc=\grD^\times\whD^1yz\wcO^\times=\grD^\times\whD^1ku\wcO^\times=k\grD^\times\whD^1\wcO^\times
u.
\end{equation}
Since $\whE$ is left invariant by $\whK^\times$ and right invariant by
$\calN(\wcO)$, we get 
\begin{equation}
  \label{eq:136}
\whE\cap [\whJ]_\scc=\whE\cap \big(k\grD^\times\whD^1\wcO^\times
u\big)=k\left(\whE\cap 
 \grD^\times\whD^1\wcO^\times \right)u=k (\whE\cap [\wcO]_\scc)u. 
\end{equation}
Lastly, observe that the map $x\mapsto kxu$ for $x\in \whE\cap [\wcO]_\scc$ induces a bijection 
\[ K^\times\bsh(\whE\cap [\wcO]_\scc)/\wcO^\times\to K^\times\bsh(\whE\cap[\whJ]_\scc)/\wcO^\times.  \]
Therefore,  (\ref{eq:132}) 
holds for $J'=\calO$ by Lemma~\ref{lem:adelic-trace}.  The proposition is proved. 
\end{proof}

\section*{Acknowledgments}
The authors would like to express their gratitude to Tomoyoshi Ibukiyama, John Voight and
Chao Zhang for stimulating discussions.   Yu is partially supported
by the MoST grants
107-2115-M-001-001-MY2 and 109-2115-M-001-002-MY3. The first draft of this 
manuscript was prepared during the first author's 2019 visit to Institute
of Mathematics, Academia Sinica. He thanks the institute for the warm
hospitality and great working conditions.

\bibliographystyle{hplain}
\bibliography{TeXBiB}
\end{document}